\documentclass[a4paper]{amsart}

\usepackage[all]{xy}
\usepackage{amsmath}
\usepackage{amssymb}
\usepackage{amsthm}
\usepackage{latexsym}
\usepackage{enumerate}
\usepackage{prettyref}
\usepackage{hyperref}
\usepackage{bm}
\usepackage{mathrsfs}

\newtheorem{theorem}{Theorem}[section]
\newrefformat{thm}{\hyperref[{#1}]{Theorem~\ref*{#1}}}
\newtheorem{definition}[theorem]{Definition}
\newrefformat{def}{\hyperref[{#1}]{Definition~\ref*{#1}}}
\newtheorem{lemma}[theorem]{Lemma}
\newrefformat{lem}{\hyperref[{#1}]{Lemma~\ref*{#1}}}
\newtheorem{proposition}[theorem]{Proposition}
\newrefformat{prop}{\hyperref[{#1}]{Proposition~\ref*{#1}}}
\newtheorem{corollary}[theorem]{Corollary}
\newrefformat{cor}{\hyperref[{#1}]{Corollary~\ref*{#1}}}
\newtheorem{remark}[theorem]{Remark}
\newrefformat{rem}{\hyperref[{#1}]{Remark~\ref*{#1}}}
{
\newtheorem{examplecore}[theorem]{Example}}
\newrefformat{ex}{\hyperref[{#1}]{Example~\ref*{#1}}}

\newenvironment{example}{\begin{examplecore}}{\hspace*{\fill}
$\square$\par\vspace{.1cm}\end{examplecore}}

\newcommand{\op}{\operatorname}
\newcommand{\et}{\mathrm{\acute{e}t}}

\begin{document}

\title[Motivic obstructions to Witt cancellation]{On motivic obstructions to Witt cancellation for quadratic forms over schemes}

\author{Matthias Wendt}

\date{October 2018}

\address{Matthias Wendt, Universit\"at Osnabr\"uck, Institut f\"ur Mathematik, Albrechtstra\ss{}e 28a, 49076 Osnabr\"uck, Germany} 
\email{m.wendt.c@gmail.com}

\subjclass[2010]{14F42, 11E12 (11E81, 19G05)}
\keywords{quadratic forms, stably hyperbolic, $\mathbb{A}^1$-homotopy, obstruction theory, Euler classes}

\begin{abstract}
The paper provides computations of the first non-vanishing $\mathbb{A}^1$-homotopy sheaves of the orthogonal Stiefel varieties which are relevant for the unstable isometry classification of quadratic forms  over smooth affine scheme over perfect fields of characteristic $\neq 2$. Together with the $\mathbb{A}^1$-representability for quadratic forms, this provides the first obstructions for rationally trivial quadratic forms to split off a hyperbolic plane. For even-rank quadratic forms, this first obstruction is a refinement of the Euler class of Edidin and Graham. A couple of consequences are discussed, such as improved splitting results over algebraically closed base fields as well as examples where the obstructions are nontrivial.  
\end{abstract}

\maketitle
\setcounter{tocdepth}{1}
\tableofcontents

\section{Introduction}

The present paper continues the investigation of analogues of Witt cancellation for stably trivial quadratic forms over smooth affine schemes which was started in \cite{hyperbolic-dim3}, originally motivated by the MathOverflow question 166249 ``Metabolic vs stably metabolic'' by K.J. Moi. In the previous paper, some low-dimensional results and examples were discussed; this time, the focus is on higher-dimensional schemes. The question which will be investigated in the paper is if a stably hyperbolic or stably trivial quadratic form splits off a hyperbolic plane. Over a field, stably hyperbolic implies hyperbolic by Witt cancellation. Over schemes, this is no longer true: stably trivial forms don't necessarily split off hyperbolic planes and it is interesting to know what kind of invariants can detect this failure. For generically trivial forms over smooth affine schemes over perfect fields of characteristic $\neq 2$, such questions can, via the  $\mathbb{A}^1$-representability theorems of \cite{gbundles2}, be translated into questions of $\mathbb{A}^1$-obstruction theory. Eventually, the above question is a question about reduction of structure group along the stabilization morphism $\op{SO}(n)\to\op{SO}(n+2)$. Hence the relevant spaces controlling the splitting questions are the orthogonal Stiefel varieties $\op{V}_{2,n+2}:=\op{SO}(n+2)/\op{SO}(n)$. The goal of the present paper is then to compute some $\mathbb{A}^1$-homotopy sheaves of these varieties, investigate the corresponding obstruction groups and deduce some results on splitting of quadratic forms.

\subsection{Splitting results}
First, let me describe more precisely the results on $\mathbb{A}^1$-homotopy sheaves of orthogonal Stiefel varieties and the consequences for the splitting questions. A first easy result one can prove is that the orthogonal Stiefel varieties $\op{V}_{2,2n}$ and $\op{V}_{2,2n+1}$ are $\mathbb{A}^1$-$(n-2)$-connected, as one would expect from real realization. A direct consequence of this is the following general splitting theorem. This is a quadratic-form version of a similar result of Serre on projective modules; for stably hyperbolic forms, it was proved by Roy \cite{roy} in greater generality. The $\mathbb{A}^1$-topology approach allows to provide a proof very close to classical algebraic-topological splitting results, cf.~Theorem~\ref{thm:stablesplitting}.

\begin{theorem}
\label{thm:roy}
Let $F$ be a perfect field of characteristic unequal to $2$, let $X=\op{Spec}A$ be a smooth affine scheme of dimension $d$ over $F$ and let $(\mathscr{P},\phi)$ be a generically trivial quadratic form over $A$ of rank $2n$ or $2n+1$. If $d\leq n-1$, then the form $(\mathscr{P},\phi)$ splits off a hyperbolic plane.
\end{theorem}

The first obstruction that appears when we move toward the unstable range is described in the following. This is a quadratic-form version of Morel's theory of the Euler class, cf.~\cite{MField}. The following main result is proved in Propositions~\ref{prop:stiefel1} and \ref{prop:stiefel2} as well as Theorems~\ref{thm:euler1} and \ref{thm:euler2}. 

\begin{theorem}
Let $F$ be a perfect field of characteristic unequal to $2$, let $X=\op{Spec}A$ be a smooth affine scheme of dimension $d$ over $F$. 

\begin{enumerate}
\item The first non-vanishing $\mathbb{A}^1$-homotopy sheaf of $\op{V}_{2,2d+1}$ is 
\[
\bm{\pi}^{\mathbb{A}^1}_{d-1}(\op{V}_{2,2d+1})\cong \mathbf{K}^{\op{MW}}_d/(1+\langle(-1)^d\rangle). 
\]
Consequently, a generically trivial quadratic form  $(\mathscr{P},\phi)$ over $A$ of rank $2d+1$ which admits a spin lift splits off a hyperbolic plane if and only if an obstruction class in the following cohomology groups vanishes: 
\[
\op{H}^d_{\op{Nis}}(X,\mathbf{K}^{\op{MW}}_d/(1+\langle(-1)^d\rangle))\cong \left\{
\begin{array}{ll} \op{H}^d_{\op{Nis}}(X,\mathbf{K}^{\op{MW}}_d/2) & d\equiv 0\bmod 2\\
\op{H}^d_{\op{Nis}}(X,\mathbf{I}^d) & d\equiv 1\bmod 2.\end{array}\right.
\]
\item The first non-vanishing $\mathbb{A}^1$-homotopy sheaf of $\op{V}_{2,2d}$ is
\[
\bm{\pi}^{\mathbb{A}^1}_{d-1}(\op{V}_{2,2d})\cong \mathbf{K}^{\op{MW}}_d\times \mathbf{K}^{\op{MW}}_{d-1}. 
\]
Consequently, a generically trivial form $(\mathscr{P},\phi)$ over $A$ of rank $2d$ which admits a spin lift splits off a hyperbolic plane if and only if an obstruction class in the following product of cohomology groups vanishes: 
\[
\widetilde{\op{CH}}^d(X)\times \op{H}^d_{\op{Nis}}(X,\mathbf{K}^{\op{MW}}_{d-1})
\]
where the second factor has a presentation
\[
\op{H}^d_{\op{Nis}}(X,\mathbf{K}^{\op{MW}}_{d-1})\cong \op{H}^d_{\op{Nis}}(X,\mathbf{W})\cong \op{coker}\left(\beta\colon\op{CH}^{d-1}(X) \to \op{H}^d_{\op{Nis}}(X,\mathbf{I}^d)\right).
\]
\end{enumerate}
\end{theorem}

\begin{remark}
Note the two alternating patterns: the two items in the above theorem correspond to the parity in homotopy groups of complex Stiefel varieties. The two subcases in the first item repeat the parity in homotopy groups of the real Stiefel varieties. Note also that the parity in the first item of the above theorem arises since $\langle(-1)^d\rangle$ is the motivic degree of the antipodal map on the quadric $\op{Q}_{2d-1}$.
\end{remark}

The computation proceeds along the classical lines, cf. \cite{steenrod}, by writing out an explicit trivialization of the torsor $\op{SO}(n)\to\op{SO}(n+1)\to\op{Q}_{n}$ to obtain an explicit description of the connecting map in the long exact homotopy sequence. The homotopy sheaves of the Stiefel varieties are then given as cokernel of the connecting map whose degree can be computed explicitly. Actually, the computation provides similar information on some higher homotopy sheaves of orthogonal Stiefel varieties: for $\op{V}_{2,2n+1}$, Morel's Freudenthal suspension theorem can be used to show that the homotopy sheaves of orthogonal Stiefel varieties are (in some stable range) built from kernel and cokernel of multiplication by $2$ or $h$ on the corresponding homotopy sheaves of spheres, depending on parity, cf. Theorem~\ref{thm:stiefelhigh}. In the case $\op{V}_{2,2n}$, the splitting as product of homotopy sheaves of spheres is true in all degrees. In particular, some higher obstruction groups can be identified as soon as more information on $\mathbb{A}^1$-homotopy sheaves of spheres becomes available.

From the above computations of the relevant obstruction groups, we can deduce some stronger splitting results such as the following, cf. Proposition~\ref{prop:special1} and \ref{prop:special2} as well as further discussion and examples in Section~\ref{sec:examples1}: 

\begin{corollary}
Let $F$ be an algebraically closed field of characteristic unequal to $2$ and let $X=\op{Spec} A$ be a smooth affine variety of dimension $d$ over $F$. 
\begin{enumerate}
\item A generically trivial quadratic form $(\mathscr{P},\phi)$ over $A$ of rank $2d+1$ which admits a spin lift splits off a hyperbolic plane. 
\item A generically trivial form $(\mathscr{P},\phi)$ over $A$ of rank $2d$ which admits a spin lift splits off a hyperbolic plane if and only if its Edidin--Graham Euler class is trivial. 
\end{enumerate}
\end{corollary}

Beyond the Edidin--Graham Euler class statement above, the obstruction classes haven't been identified exactly. Of course, more precise results and examples could be provided if one could identify the obstruction classes explicitly in terms of characteristic classes of the quadratic forms. However, at this point, the Chow--Witt characteristic classes of the orthogonal groups are not known. As the computation of $\widetilde{\op{CH}}^\bullet({\op{B}}_{\op{Nis}}G,\mathscr{L})$ for $G=\op{O}(n)$ resp. $G=\op{SO}(n)$ appears to be significantly more complicated than for  $G=\op{Sp}_{2n},\op{GL}_n,\op{SL}_n$, the relevant computations will be subject for future research.

As a final remark, I want to mention that it is also possible to ask if a stably hyperbolic form is the hyperbolic form of a vector bundle. Again, this question reduces to Witt cancellation over fields, but over schemes, stably hyperbolic forms aren't necessarily hyperbolic. In terms of bundle classification, this is the question of reduction of structure group along the hyperbolic homomorphism $H\colon\op{GL}_n\to \op{SO}(2n)$. The relevant obstruction theory is controlled by the space $\Gamma_n=\op{SO}(2n)/\op{GL}_n$. It is easy to show that the stabilization morphism $\Gamma_n\to\Gamma_{n+1}$ induces isomorphisms on $\bm{\pi}^{\mathbb{A}^1}_i$ for $i\leq n-2$, and the homotopy sheaves in this range are given by $\bm{\pi}^{\mathbb{A}^1}_i(\Gamma_n)=\bm{\pi}^{\mathbb{A}^1}_i(\op{O}/\op{GL})\cong \mathbf{GW}^3_i$. This shows that under the conditions of Theorem~\ref{thm:roy}, a stably hyperbolic form is already the hyperbolic form associated to some vector bundle (which is also covered by the results of Roy \cite{roy}). Computations of the first unstable homotopy sheaf could probably be done along the lines of the computations of Harris \cite{harris} and Massey \cite{massey}; note that the topological analogue of the hyperbolicity question above is the classification of almost complex structures on real vector bundles of even rank.

\subsection{Structure of the paper} 
We start with some preliminaries on quadratic forms and orthogonal groups in Section~\ref{sec:prelims}. The computation of the $\mathbb{A}^1$-homotopy of the Stiefel varieties requires the discussion of an orthogonal clutching construction, explicit trivializations of $\op{SO}(2n+1)$-torsors and computations of degrees of maps between motivic spheres in Section~\ref{sec:clutching}. This leads to the description of some $\mathbb{A}^1$-homotopy sheaves for orthogonal Stiefel varieties in Section~\ref{sec:stiefel};  consequences for splitting theorems and relevant examples are then given in Section~\ref{sec:splitting}. The appendix~\ref{sec:stabilization} also discuss some structural statements concerning the first unstable $\mathbb{A}^1$-homotopy sheaves of the orthogonal groups $\op{SO}(2n-1)$.

\section{Preliminaries on quadratic forms and orthogonal groups}
\label{sec:prelims}

In this paper, $F$ always denotes a perfect field of characteristic $\neq 2$. We consider smooth affine schemes $X=\op{Spec} A$ over $F$ and are interested in classification results for quadratic forms over the ring $A$. In the present section, we provide a short recollection on quadratic forms, orthogonal groups and their associated quadrics and Stiefel varieties. We also shortly discuss the $\mathbb{A}^1$-representability results for orthogonal groups and the resulting $\mathbb{A}^1$-obstruction theory approach to splitting questions.

\subsection{Quadratic forms as torsors}

Most of the material concerning quadratic forms can be found in standard textbooks on the subject, such as \cite{knus} (or \cite[Sections 3,4]{roy} which also deals with splitting questions similar to the ones in the present paper). A similar recollection was already used in \cite{hyperbolic-dim3}.

\begin{definition}
Let $F$ be a field of characteristic unequal to $2$. 
\begin{itemize}
\item A \emph{quadratic form} over a commutative $F$-algebra $A$ is given by a finitely generated projective $A$-module $\mathscr{P}$ together with a map $\phi\colon\mathscr{P}\to A$ such that for each $a\in A$ and $x\in\mathscr{P}$, we have $\phi(ax)=a^2\phi(x)$ and $B_\phi(x,y)=\phi(x+y)-\phi(x)-\phi(y)$ is a symmetric bilinear form $B_\phi\in \op{Sym}^2(\mathscr{P}^\vee)$. 
\item 
The \emph{rank} of the quadratic form is defined to be rank of the projective module $\mathscr{P}$. 
\item A quadratic form $(\mathscr{P},\phi)$ is \emph{non-singular} or \emph{non-degenerate} if the morphism $\mathscr{P}\to \mathscr{P}^\vee\colon x\mapsto B_\phi(x,-)$ is an isomorphism. 
\item An element $x\in\mathscr{P}$ is called \emph{isotropic} if $\phi(x)=0$. 
\item A morphism $f\colon (\mathscr{P}_1,\phi_1)\to (\mathscr{P}_2,\phi_2)$ of quadratic forms is an $A$-linear map $f\colon \mathscr{P}_1\to\mathscr{P}_2$ such that $\phi_2(f(x))=\phi_1(x)$ for all $x\in\mathscr{P}_1$. An isomorphism of quadratic forms is also called an \emph{isometry}. The automorphism group of a quadratic form is called the \emph{orthogonal group} of the quadratic form. 
\item Given two quadratic forms $(\mathscr{P}_1,\phi_1)$ and $(\mathscr{P}_2,\phi_2)$, there is a quadratic form 
\[
(\mathscr{P}_1,\phi_1)\perp(\mathscr{P}_2,\phi_2):=(\mathscr{P}_1\oplus\mathscr{P}_2, \phi_1+\phi_2)
\]
which is called the \emph{orthogonal sum}.
\end{itemize}
\end{definition}

\begin{example}
Let $A$ be a commutative ring and let $\mathscr{P}$ be a finitely generated projective module. Then there is a quadratic form whose underlying module is $\mathscr{P}\oplus \mathscr{P}^\vee$, equipped with the evaluation form $\op{ev}\colon(x,f)\mapsto f(x)$. The quadratic form $(\mathscr{P}\oplus\mathscr{P}^\vee,\op{ev})$ is called the \emph{hyperbolic space} associated to the projective module $\mathscr{P}$. In the special case where $\mathscr{P}=A$ is the free module of rank 1, this is called the \emph{hyperbolic plane} $\mathbb{H}$ over $A$. 
\end{example}

The following are the standard hyperbolicity notions from quadratic form theory, cf.~\cite[Section VIII.2]{knus}.

\begin{definition}
A quadratic form is called \emph{hyperbolic}, if it is isometric to $\mathbb{H}(\mathscr{P})$ for some projective module $\mathscr{P}$.  A quadratic form $(\mathscr{P},\phi)$ is called \emph{stably hyperbolic} if there exists a projective module $\mathscr{Q}$ such that $(\mathscr{P},\phi)\perp\mathbb{H}(\mathscr{Q})$ is hyperbolic. A quadratic form $(\mathscr{P},\phi)$ over an integral domain $A$ is called \emph{rationally hyperbolic} if $(\mathscr{P},\phi)\otimes_A\op{Frac}(A)$ is hyperbolic.
\end{definition}

\begin{remark}
Note that stably hyperbolic forms are then necessarily of even rank, and stably hyperbolic forms are those that become $0$ in the Witt ring.
\end{remark}

For the purposes of the present paper, we will also be interested in stricter notions of stable triviality of quadratic forms:

\begin{definition}
A quadratic form $(\mathscr{P},\phi)$ is called \emph{stably trivial} if it becomes isometric to one of the split forms $\mathbb{H}^{\perp n}$ or $\mathbb{H}^{\perp n}\perp (A,a\mapsto a^2)$ after adding sufficiently many hyperbolic planes. 
\end{definition}

The stably trivial forms are those which represent classes in $\mathbb{Z}\cdot [\mathbb{H}]\oplus\mathbb{Z}\cdot [(A,a\mapsto a^2)]\subseteq \op{GW}(A)$ in the Grothendieck--Witt ring of $A$. 

The classical Witt cancellation theorem implies in particular that the notions of hyperbolic, stably trivial and stably hyperbolic all agree over fields: 

\begin{proposition}[Witt cancellation theorem]
Let $k$ be a field of characteristic $\neq 2$ and let $(V_1,\phi_1)$ and $(V_2,\phi_2)$ be two quadratic forms over $k$. If $(V_1,\phi_1)\oplus\mathbb{H}\cong (V_2,\phi_2)\oplus\mathbb{H}$ then $(V_1,\phi_1)\cong(V_2,\phi_2)$. In particular, a stably hyperbolic form is hyperbolic. 
\end{proposition}

\begin{remark}
\label{rem:abuse2}
All quadratic forms over fields considered in this paper are hyperbolic or of the form $\mathbb{H}^{\perp n}\oplus(A,a\mapsto a^2)$. In abuse of notation, the group $\op{SO}(n)$ will always denote the special orthogonal group associated to the \emph{split form} of rank $n$. I apologize to anyone who might be offended by this.
\end{remark}

Now we recall the representability theorem for torsors from \cite{gbundles2}. The identification of quadratic forms with torsors for the orthogonal groups is classical, cf. \cite[Section 5]{RojasVistoli} or \cite[Section 2]{hyperbolic-dim3}.

\begin{theorem}
\label{thm:representability}
Let $k$ be a field, and let $X=\op{Spec} A$ be a smooth affine $k$-scheme. Let $G$ be a reductive group such that each absolutely almost simple component of $G$ is isotropic. Then there is a bijection 
\[
\op{H}^1_{\op{Nis}}(X;G)\cong [X, {\op{B}}_{\op{Nis}}G]_{\mathbb{A}^1}
\]
between the pointed set of isomorphism classes of rationally trivial $G$-torsors over $X$ and the pointed set of $\mathbb{A}^1$-homotopy classes of maps $X\to {\op{B}}_{\op{Nis}}G$.
\end{theorem}

As a consequence, for the split group $\op{SO}(n)$ and a smooth affine scheme $X=\op{Spec} A$, we get a natural bijection between the pointed set of generically split quadratic forms over the ring $A$ and the pointed set of homotopy classes of maps $[X,\op{B}_{\op{Nis}}\op{SO}(n)]_{\mathbb{A}^1}$.\footnote{Note that we are talking about \emph{unpointed} maps here.} There is a similar statement for the spin groups which are $\mathbb{A}^1$-connected. In this case, the corresponding classifying spaces are $\mathbb{A}^1$-simply connected and consequently there is a canonical bijection between pointed maps $X_+\to\op{B}_{\op{Nis}}\op{Spin}(n)$ and unpointed maps $X\to\op{B}_{\op{Nis}}\op{Spin}(n)$. This is not true for the special orthogonal groups, where the sheaf of connected components is the Nisnevich sheaf $\mathcal{H}^1_{\et}(\mu_2)$; the unpointed classification is given by taking a quotient of the pointed result by the action of the fundamental group sheaf, cf. \cite{hyperbolic-dim3} for more details.


\subsection{Split orthogonal groups}
\label{sec:orthogonal}

For some of the computations, we will need an explicit description of the split orthogonal groups. 

For the orthogonal group $\op{SO}(2n)$ on an even-dimensional vector space,  the set of variables is given by $(X_1,\dots,X_n,X_{-n},\dots,X_{-1})$, and the symmetric bilinear form of rank $2n$ is given by $B(u,v)=u^{\op{t}}Jv$ where $J$ is the matrix with $J_{\alpha,\beta}=\delta_{\alpha,-\beta}$. The associated quadratic form is $\sum_{i=1}^nX_iX_{-i}$. The orthogonal group $\op{O}(2n)$ is the subgroup of $\op{GL}_{2n}$ consisting of the matrices $A$ satisfying $AJA^{\op{t}}=J$. The special orthogonal group $\op{SO}(2n)$ is given as the intersection $\op{SO}(2n)=\op{SL}_{2n}\cap\op{O}(2n)$ inside $\op{GL}_{2n}$.  

For the orthogonal group $\op{SO}(2n+1)$, we use the following concrete realization of the split symmetric bilinear form of rank $2n+1$, cf. \cite[Section 1]{vavilov}. The set of variables is given by  $(X_1,\dots,X_n,X_0,X_{-n},\dots,X_{-1})$. Consider the matrix 
\[
J_{\alpha,\beta}=\left\{\begin{array}{ll}
\frac{1}{2} & \alpha=-\beta\neq 0\\
1 & \alpha=\beta=0\\
0 & \textrm{ otherwise}
\end{array}\right.
\]
The orthogonal group $\op{O}(2n+1)$ is the subgroup of $\op{GL}_{2n+1}$ consisting of the matrices $A$ satisfying $AJA^{\op{t}}=J$. The special orthogonal subgroup $\op{SO}(2n+1)$ is given as the intersection $\op{SO}(2n+1)=\op{SL}_{2n+1}\cap \op{O}(2n+1)$ inside $\op{GL}_{2n+1}$. 

The orthogonal group $\op{O}(2n+1)$ preserves the bilinear form $B(u,v)=u^{\op{t}}Jv$, resp. the quadratic form $X_0^2+\sum_{i=1}^n X_iX_{-i}$. The inner product of the standard basis vectors $(\op{e}_1,\dots,\op{e}_n,\op{e}_0,\op{e}_{-n},\dots,\op{e}_{-1})$ is given by
\[
(\op{e}_i,\op{e}_j)=\delta_{i,-j}+\delta_{i,0}\delta_{j,0}.
\]
An orthogonal basis (with vectors of length $1$ or $-1$) is given by 
\[
v_{\pm i} =(\op{e}_i\pm \op{e}_{-i}), \qquad v_0=\op{e}_0.
\]
A vector $u=(t_1,\dots,t_n,t_0,t_{-n},\dots,t_{-1})$ can be expressed in the orthogonal basis as follows:
\[
u=t_0v_0+\sum_{i=1}^n\left(\left(\frac{t_i+t_{-i}}{2}\right)v_i+ \left(\frac{t_i-t_{-i}}{2}\right)v_{-i}\right).
\]

In the above orthogonal groups, we have the natural (Householder) reflections: if $w$ is a vector which is not isotropic, then the reflection at the hyperplane perpendicular to $w$ is given by
\[
\tau_w(x)=x-\frac{2 B(w,x)}{B(w,w)}w.
\]
If $u,v$ are vectors in $V$ such that $w:=u-v$ is not isotropic, the above reflection for $w$ maps $u$ to $v$. 

\subsection{Quadrics and Stiefel varieties}
\label{sec:quadrics}

We recall some facts about smooth affine split quadrics and Stiefel varieties. First, the odd-dimensional smooth affine split quadrics are defined as follows:
\[
\op{Q}_{2n-1}:=\op{Spec}k[X_1,\dots,X_n,Y_1,\dots,Y_n]/(\sum X_iY_i-1)
\]

For the even-dimensional quadrics, there are two possible presentations over fields of characteristic $\neq 2$. One definition, the one used e.g. in \cite{AsokDoranFasel} provides a scheme definable over the integers: 
\[
\op{Q}_{2n}':=\op{Spec} k[X'_1,\dots,X'_n,Y'_1,\dots,Y'_n,Z']/(\sum X'_iY'_i-Z'(1+Z'))
\]
However, it seems that this quadric is not globally a homogeneous space for the split orthogonal groups over $\mathbb{Z}$. Since the later computations and constructions require the relation to the orthogonal groups, we will use the following presentation of the smooth split affine quadrics, defined by  the quadratic form $q$ used to define the explicit model of $\op{SO}(2n+1)$ discussed above: 
\[
\op{Q}_{2n}:=\op{Spec} k[X_1,\dots,X_n,Y_1,\dots,Y_n,Z]/(\sum X_iY_i+Z^2-1). 
\]

An isomorphism between the quadrics $\op{Q}_{2n}$ and $\op{Q}_{2n}'$ is given as follows, cf. \cite{AsokDoranFasel}\footnote{Note that the explicit isomorphism is only found in some earlier arXiv version, not the published version of the paper.}: 
\[
\op{Q}_{2n}'\to \op{Q}_{2n}\colon  X'_i\mapsto -X_i/2, \, Y'_i\mapsto Y_i/2, \, Z' \mapsto (Z-1)/2.
\]
This morphism takes the form $\sum_iX'_iY'_i-Z'(1+Z')$ to
\[
\frac{1}{4}\left(-\sum_iX_iY_i-(Z-1)(Z+1)\right)= \frac{1}{4}\left(-\sum_iX_iY_i-Z^2+1\right).
\]
Under this isomorphism, the principal open subsets $\op{D}(Z')$ and $\op{D}(1+Z')$ appearing in the clutching construction of \cite{AsokDoranFasel} are mapped to the principal open subsets $\op{D}(Z-1)$ and $\op{D}(Z+1)$ in the quadric $\op{Q}_{2n}$. 

The natural choice of base point $v$ for $\op{Q}_{2n-1}$ is given by $X_1=Y_1=1$ and $X_i=Y_i=0$ for $i\geq 2$. Similarly, the natural choice of base point $v$ for $\op{Q}_{2n}$ is given by $Z=1$ and $X_i=Y_i=0$ for all $i$. With these choices of base point, we have natural projections
\[
\pi\colon \op{SO}(m)\to\op{Q}_{m-1}\colon A\mapsto A\cdot v
\]
which induce isomorphisms $\op{Q}_{m-1}\cong \op{SO}(m)/\op{SO}(m-1)$. 


Consider the stabilization homomorphism $\op{SO}(n-2)\hookrightarrow\op{SO}(2)$. The corresponding quotient $\op{V}_{2,n}:=\op{SO}(n)/\op{SO}(n-2)$ is one example of a \emph{Stiefel variety}. Since homogeneous spaces $\op{GL}_n/\op{GL}_{n-k}$ have been already been called Stiefel varieties in the motivic homotopy literature, we will use the distinction \emph{orthogonal Stiefel varieties}. There is a fiber bundle
\[
\op{SO}(n-1)/\op{SO}(n-2)\to \op{SO}(n)/\op{SO}(n-2)\to \op{SO}(n)/\op{SO}(n-1).
\]
Via the above identifications of $\op{SO}(n)/\op{SO}(n-1)$ as smooth affine quadrics, this allows to write the orthogonal Stiefel varieties as sphere bundles over a sphere:
\[
\op{Q}_{n-2}\to \op{V}_{2,n}\to \op{Q}_{n-1}.
\]

Via the representability results of \cite{gbundles2}, the above geometric statements provide $\mathbb{A}^1$-fiber sequences which we will use to describe some $\mathbb{A}^1$-homotopy sheaves of orthogonal Stiefel varieties. 

First of all, there is an $\mathbb{A}^1$-fiber sequence related to stabilization and splitting off hyperbolic planes. If $X\to {\op{B}}_{\op{Nis}}\op{SO}(n)$ classifies a generically trivial quadratic form, the form splits off a hyperbolic plane if and only if the classifying map lifts through ${\op{B}}_{\op{Nis}}\op{SO}(n-2)\to {\op{B}}_{\op{Nis}}\op{SO}(n)$. By the below fiber sequence, the obstructions to splitting off hyperbolic planes will then live in the cohomology with coefficients in $\bm{\pi}^{\mathbb{A}^1}_i(\op{V}_{2,n})$. 

\begin{proposition}
\label{prop:stabil}
Let $F$ be a perfect field of characteristic $\neq 2$. Under the correspondence of Theorem~\ref{thm:representability}, the morphism 
\[
{\op{B}_{\op{Nis}}}\op{SO}(n-2)\to{\op{B}_{\op{Nis}}}\op{SO}(n)
\]
induced from the standard embedding $\op{SO}(n-2)\to\op{SO}(n)$ corresponds to adding a hyperbolic plane. There is an $\mathbb{A}^1$-homotopy fiber sequence
\[
\op{V}_{2,n}\to{\op{B}}_{\op{Nis}}\op{SO}(n-2)\to {\op{B}}_{\op{Nis}}\op{SO}(n). 
\]
\end{proposition}

\begin{proof}
The $\op{SO}(n-2)$-torsor $\op{SO}(n)\to \op{V}_{2,n}$ is Zariski-locally trivial: it is the universal $\op{SO}(n-2)$-torsor which becomes trivial after adding a hyperbolic plane. By Witt cancellation, it must already be rationally trivial. The claim then follows directly from the results of \cite[Section 2.4]{gbundles2}, with an argument similar to \cite[Theorem 4.2.2]{gbundles2}. 
\end{proof}

\begin{remark}
Note that there are similar $\mathbb{A}^1$-fiber sequences $\op{Q}_n\to {\op{B}}_{\op{Nis}}\op{SO}(n) \to{\op{B}}_{\op{Nis}}\op{SO}(n+1)$, cf. \cite[Theorems 4.2.1 and 4.2.2]{gbundles2}. 
\end{remark}

The easiest way to understand the  $\mathbb{A}^1$-homotopy sheaves of the Stiefel varieties $\op{V}_{2,n}$ is the following fiber sequence: 

\begin{lemma}
\label{lem:quadfib}
Let $F$ be a perfect field of characteristic $\neq 2$. There is an $\mathbb{A}^1$-fiber sequence $\op{Q}_{n-2}\to \op{V}_{2,n}\to \op{Q}_{n-1}$, and consequently there is a long exact sequence of $\mathbb{A}^1$-homotopy sheaves
\[
\cdots\to \bm{\pi}^{\mathbb{A}^1}_{i+1}(\op{Q}_{n-1})\to \bm{\pi}^{\mathbb{A}^1}_{i}(\op{Q}_{n-2})\to \bm{\pi}^{\mathbb{A}^1}_{i}(\op{V}_{2,n})\to \bm{\pi}^{\mathbb{A}^1}_{i}(\op{Q}_{n-1})\to\cdots
\]
Here the base points are induced from the standard base point on $\op{Q}_{n-2}$. 
\end{lemma}

\begin{proof}
By definition, $\op{V}_{2,n}\cong\op{SO}(n)/\op{SO}(n-2)$, and the intermediate group $\op{SO}(n-1)$ gives rise to the maps $\op{Q}_{n-2}\cong\op{SO}(n-1)/\op{SO}(n-2)\hookrightarrow \op{V}_{2,n}$ and $\op{V}_{2,n}\twoheadrightarrow \op{SO}(n)/\op{SO}(n-1)\cong\op{Q}_{n-1}$. Zariski-locally, the $\op{Q}_{n-2}$-bundle $\op{V}_{2,n}\to\op{Q}_{n-1}$ is trivial; and it is the associated bundle for the $\op{SO}(n-1)$-torsor $\op{SO}(n)\to \op{Q}_{n-1}$ and the natural transitive $\op{SO}(n-1)$-action on $\op{Q}_{n-2}$, by definition. From this, using the results on associated bundles from \cite{torsors}, the claimed fiber sequence is obtained via pullback from the fiber sequence for $\op{Q}_{n-2}$ as follows:
\[
\xymatrix{
\op{Q}_{n-2}\ar[r] \ar[d] & \op{Q}_{n-2} \ar[d] \\
\op{V}_{2,n} \ar[r] \ar[d] & {\op{B}}_{\op{Nis}}\op{SO}(n-2) \ar[d] \\
\op{Q}_{n-1} \ar[r] & {\op{B}}_{\op{Nis}}\op{SO}(n-1)
}
\]
The claim about $\mathbb{A}^1$-homotopy sheaves follows directly. 
\end{proof}

\subsection{Obstruction theory}

We recall some basics of relative obstruction theory and Moore--Postnikov factorizations which are used for the obstruction-theoretic approach to torsor classification in $\mathbb{A}^1$-homotopy. Most of the material here is taken from \cite{AsokFaselSplitting}. 

Suppose that $p\colon(\mathcal{E},x)\to (\mathcal{B},y)$ is a pointed map of $\mathbb{A}^1$-connected spaces and let $\mathcal{F}$ be the $\mathbb{A}^1$-homotopy fiber of $p$. Assume that $p$ is an $\mathbb{A}^1$-fibration, $\mathcal{B}$ is $\mathbb{A}^1$-local and that $\mathcal{F}$ is $\mathbb{A}^1$-simply connected. Then there are pointed spaces $(\mathcal{E}^{(i)},x_i)$, $i\in\mathbb{N}$  $\mathcal{E}^{(0)}=\mathcal{B}$, pointed morphisms 
\[
g^{(i)}\colon \mathcal{E}\to\mathcal{E}^{(i)}, \qquad 
h^{(i)}\colon \mathcal{E}^{(i)}\to \mathcal{B}, \qquad 
p^{(i)}\colon\mathcal{E}^{(i+1)}\to \mathcal{E}^{(i)}
\]
and commutative diagrams
\[
\xymatrix{
& \mathcal{E}^{(i+1)} \ar[rd]^{h^{(i+1)}} \ar[d]_{p^{(i)}} \\
\mathcal{E} \ar[ru]^{g^{(i+1)}} \ar[r]_{g^{(i)}} &\mathcal{E}^{(i)} \ar[r]_{h^{(i)}} & \mathcal{B}
}
\]
such that the following statements are satisfied: 
\begin{enumerate}
\item For any $i$, we have $h^{(i)}\circ g^{(i)}=p$. 
\item The morphism $\bm{\pi}^{\mathbb{A}^1}_n(\mathcal{E}) \to \bm{\pi}^{\mathbb{A}^1}_n(\mathcal{E}^{(i)})$ induced by $g^{(i)}$ is an isomorphisms for $n\leq i$ and an epimorphism for $n=i+1$.
\item The morphism $\bm{\pi}^{\mathbb{A}^1}_n(\mathcal{E}^{(i)})\to \bm{\pi}^{\mathbb{A}^1}_n(\mathcal{B})$ induced by $h^{(i)}$ is an isomorphism for $n>i+1$ and a monomorphism for $n=i+1$. 
\item The induced morphism $\mathcal{E}\to \op{holim}_i\mathcal{E}^{(i)}$ is an $\mathbb{A}^1$-weak equivalence. 
\end{enumerate}
For any $i$, there is an $\mathbb{A}^1$-fiber sequence
\[
\op{K}(\bm{\pi}_i^{\mathbb{A}^1}(\mathcal{F}),i)\to \mathcal{E}^{(i+1)}\xrightarrow{p^{(i)}} \mathcal{E}^{(i)}
\]
Moreover, the $p^{(i)}$ are twisted principal $\mathbb{A}^1$-fibrations, in the sense that there is a unique morphism (up to $\mathbb{A}^1$-homotopy)
\[
k_{i+1}\colon\mathcal{E}^{(i)}\to \op{K}^{\bm{\pi}_1^{\mathbb{A}^1}(\mathcal{B})}(\bm{\pi}_i^{\mathbb{A}^1}(\mathcal{F}),i+1),
\]
called the $k$-invariant, such that there is an $\mathbb{A}^1$-homotopy pullback square
\[
\xymatrix{
\mathcal{E}^{(i+1)} \ar[r] \ar[d] & \op{B}\bm{\pi}^{\mathbb{A}^1}_1(\mathcal{B}) \ar[d] \\
\mathcal{E}^{(i)} \ar[r]_{k_{i+1}} & \op{K}^{\bm{\pi}^{\mathbb{A}^1}_1(\mathcal{B})}(\bm{\pi}^{\mathbb{A}^1}_i(\mathcal{F}),i+1).
}
\]
If the base space $\mathcal{B}$ is $\mathbb{A}^1$-simply-connected, this reduces to an $\mathbb{A}^1$-fiber sequence 
\[
\mathcal{E}^{(i+1)}\to \mathcal{E}^{(i)} \xrightarrow{k_{i+1}}  \op{K}(\bm{\pi}^{\mathbb{A}^1}_i(\mathcal{F}),i+1).
\]

For a smooth scheme $X$ with a morphism $f\colon X\to \mathcal{B}$, the above Moore--Postnikov factorization provides a sequence of obstructions to lifting $f$ along $p\colon\mathcal{E}\to\mathcal{B}$: if we have already lifted to $\mathcal{E}^{(i)}$, then a lift to $\mathcal{E}^{(i+1)}$ exists if and only if the composition 
\[
X\to \mathcal{E}^{(i)}\xrightarrow{k_{i+1}} \op{K}^{\bm{\pi}^{\mathbb{A}^1}_1(\mathcal{B})}(\bm{\pi}^{\mathbb{A}^1}_i(\mathcal{F}), i+1)
\]
lifts to $X\to\op{B}\bm{\pi}^{\mathbb{A}^1}_1(\mathcal{B})$. If $\mathcal{B}$ is $\mathbb{A}^1$-simply-connected, then the obstruction is simply a cohomology class in $\op{H}^{i+1}_{\op{Nis}}(X,\bm{\pi}^{\mathbb{A}^1}_i(\mathcal{F}))$, and the set of lifts $X\to\mathcal{E}^{(i+1)}$ is then given as a quotient of $\op{H}^i_{\op{Nis}}(X,\bm{\pi}^{\mathbb{A}^1}_i(\mathcal{F}))$ by the image of the looping of the $k$-invariant. In the more general case, the obstructions and liftings are  elements of some $\bm{\pi}^{\mathbb{A}^1}_1(\mathcal{B})$-equivariant cohomology groups, cf. \cite[Section 6]{AsokFaselSplitting} for more details. Note for smooth schemes there are only finitely many obstructions since all obstructions above the dimension of the scheme vanish, for Nisnevich cohomological dimension reasons.

We shortly discuss the specific case of quadratic forms. First, we note that there is a sequence of $\mathbb{A}^1$-fiber sequences
\[
\xymatrix{
\op{V}_{2,n+2}\ar[r] \ar[d]_= & {\op{B}}_{\op{Nis}}\op{Spin}(n) \ar[r] \ar[d] & {\op{B}}_{\op{Nis}}\op{Spin}(n+2) \ar[d]\\
\op{V}_{2,n+2} \ar[r] & {\op{B}}_{\op{Nis}}\op{SO}(n) \ar[r] & {\op{B}}_{\op{Nis}}\op{SO}(n+2)
}
\]
Since the classifying spaces of the spin groups are $\mathbb{A}^1$-simply-connected, the obstruction groups for splitting off a trivial $\op{Spin}(2)$-torsor from a $\op{Spin}(n+2)$-torsor on a smooth scheme $X$ are given by $\op{H}^{i+1}_{\op{Nis}}(X,\bm{\pi}^{\mathbb{A}^1}_i(\op{V}_{2,n+2}))$.

On the other hand, the classifying spaces of the orthogonal groups  have a non-trivial $\mathbb{A}^1$-fundamental group, given by $\mathscr{H}^1_{\et}(\mu_2)$. We need to discuss the action of $\mathscr{H}^1_{\et}(\mu_2)$ on the $\mathbb{A}^1$-homotopy sheaves of the classifying spaces  ${\op{B}}_{\op{Nis}}\op{SO}(n)$. Note that over a field $K$, $\op{H}^1_{\et}(K,\mu_2)=K^\times/(K^\times)^2$. To understand the action of the square residues on the sections of $\bm{\pi}_i^{\mathbb{A}^1}{\op{B}}_{\op{Nis}}\op{SO}(n)$ one can follow the arguments in \cite[Section 6.2]{AsokFaselSplitting}. This allows to identify the action of the square residues as the usual action of square classes on strictly $\mathbb{A}^1$-invariant sheaves, given by restricting the Milnor--Witt module structure along the universal morphism $\mathbb{G}_{\op{m}}/2\to \mathbf{K}^{\op{MW}}_0$. 

If $X$ is a smooth $F$-scheme and $f:X\to {\op{B}}_{\op{Nis}}\op{SO}(n)$ is a map, then  the obstruction to lifting $f$ through the universal covering  ${\op{B}}_{\op{Nis}}\op{Spin}(n)\to {\op{B}}_{\op{Nis}}\op{SO}(n)$ is a torsor over $X$ with structure group $\mathbb{G}_{\op{m}}/2$, the spinor norm torsor, which can be concretely described as follows. Assume we have a Nisnevich covering $U_i\to X$ trivializing the $\op{SO}(n)$-torsor. The transition maps are morphisms $U_i\times_X U_j\to \op{SO}(n)$. We can compose these with the spinor norm $\op{SO}(n)\to \mathbb{G}_{\op{m}}/2$, and this provides the transition function for the $\mathbb{G}_{\op{m}}/2$-torsor over $X$. If this torsor is trivial, then the maps $X\to \op{K}^{\mathscr{H}^1_{\et}(\mu_2)}(\bm{\pi}^{\mathbb{A}^1}_i(\op{V}_{2,n+2}),i+1)$, corresponding to $\mathscr{H}^1_{\et}(\mu_2)$-equivariant cohomology of $X$ with coefficients in $\bm{\pi}^{\mathbb{A}^1}_i(\op{V}_{2,n+2})$ can be identified with ``ordinary'' cohomology groups $\op{H}^{i+1}_{\op{Nis}}(X,\bm{\pi}^{\mathbb{A}^1}_i(\op{V}_{2,n+2}))$. 

By the discussion in \cite[Section 6]{hyperbolic-dim3}, stably trivial quadratic forms always have spin lifts; and we will only focus on rationally trivial quadratic forms which have a spin lift in this paper. In particular, there is no need to deal with equivariance for the fundamental group $\mathscr{H}^1_{\et}(\mu_2)$ of the classifying spaces ${\op{B}}_{\op{Nis}}\op{SO}(n)$ for the present investigation. Put differently, to determine if a quadratic form splits off a hyperbolic plane, we can choose a spin lift and then use the above obstruction groups to determine if it splits off a trivial $\op{Spin}(2)$-torsor. 

\section{Orthogonal clutching and connecting map}
\label{sec:clutching}

For the later computation of the first non-vanishing $\mathbb{A}^1$-homotopy sheaf of the orthogonal Stiefel varieties, we need to identify the $\op{SO}(2n)$-torsor $\op{SO}(2n+1)\to\op{Q}_{2n}$ in terms of a clutching construction for a morphism $\op{Q}_{2n-1}\to \op{SO}(2n)$. This is obtained by an explicit trivialization of the torsor $\op{SO}(2n+1)\to\op{Q}_{2n}$ similar to the classical results in \cite[Section 23]{steenrod}. We also compute the degree of the composition $\op{Q}_{2n-1}\to\op{SO}(2n)\to\op{Q}_{2n-1}$ of the clutching function and the natural projection. These computations will be used in the description of the connecting map in the long exact $\mathbb{A}^1$-homotopy sequence for the Stiefel fibration $\op{Q}_{2n-1}\to\op{V}_{2,2n+1}\to\op{Q}_{2n}$ of Lemma~\ref{lem:quadfib}.

\subsection{Recollection of the clutching construction}

We shortly review the algebraic version of the clutching construction for vector bundles which was established in \cite{AsokDoranFasel}.  The relevant covering for the clutching construction for the quadric $\op{Q}_{2n}'$ (described by the equation $\sum X_i'Y_i'-Z'(1+Z')$)  is given by the principal open subschemes $\op{D}(Z')$ and $\op{D}(1+Z')$ whose intersection is $\op{D}(Z'(1+Z'))$. On the intersection, we have the morphism $\op{D}(Z'(1+Z'))\to \op{Q}_{2n-1}$ given by
\[
(X'_1,\dots,X'_n,Y'_1,\dots,Y'_n,Z')\mapsto \left(\frac{X'_1}{Z'},\dots,\frac{X'_n}{Z'},\frac{Y'_1}{1+Z'},\dots, \frac{Y'_n}{1+Z'}\right).
\]
Via the isomorphism $\op{Q}_{2n}'\to\op{Q}_{2n}$, cf. Section~\ref{sec:quadrics}, the covering above corresponds to the covering of $\op{Q}_{2n}$ given by the principal open subschemes $\op{D}(Z-1)$ and $\op{D}(Z+1)$ with intersection $\op{D}(Z^2-1)$. The natural projection $\op{D}(Z^2-1)\to\op{Q}_{2n-1}$ is given by 
\[
X_i\mapsto \frac{X_i}{1-Z}, \qquad Y_i\mapsto \frac{Y_i}{1+Z}.
\]
This induces a morphism $\op{Q}_{2n}\to\Sigma_{\op{s}}^1\op{Q}_{2n-1}$, and the map $\op{Q}_{2n}\to{\op{B}}G$ classifying a $G$-torsor over $\op{Q}_{2n}$ corresponds, via adjunction 
\[
[\op{Q}_{2n},{\op{B}}G]_{\mathbb{A}^1} \cong [\Sigma_{\op{s}}^1\op{Q}_{2n-1},{\op{B}}G]_{\mathbb{A}^1} \cong [\op{Q}_{2n-1},\Omega_s{\op{B}}G]_{\mathbb{A}^1} \cong [\op{Q}_{2n-1},G]_{\mathbb{A}^1},
\]
to a clutching function $\op{Q}_{2n-1}\to G$. Concretely, we can start with a morphism $\op{Q}_{2n-1}\to G$, and take the $G$-torsor which is trivial on the principal open subschemes $\op{D}(Z\pm 1)$ and whose transition function on $\op{D}(Z^2-1)$ is given by the composition $\op{D}(Z^2-1)\to\op{Q}_{2n-1}\to G$, cf. \cite[Theorem 4.3.6]{AsokDoranFasel}. 


For the other direction, we have the following:

\begin{proposition}
\label{prop:clutching}
Let $G$ be an isotropic group, and let $\mathscr{E}$ be a rationally trivial $G$-torsor classified by $\op{Q}_{2n}\to {\op{B}}G$.  Assume that we have obtained trivializations over the open subschemes $\op{D}(Z\pm 1)$ and a corresponding transition function $\op{D}(Z^2-1)\to G$. The composition $\op{Q}_{2n-1}\to\op{D}(Z^2-1)\to G$ of the inclusion via $Z=0$ and the transition function is a clutching function for the $G$-torsor $\mathscr{E}$.
\end{proposition}

\begin{proof}
 The morphism $\op{Q}_{2n-1} \to \op{D}(Z^2-1) \to \op{Q}_{2n-1}$ which first includes the quadric $\op{Q}_{2n-1}$ via $Z=0$ and then applies the standard projection is the identity. By the representability theorem of \cite{gbundles2}, the isomorphism classes of rationally trivial $G$-torsors over $\op{Q}_{2n}$ are in natural bijection with the naive $\mathbb{A}^1$-homotopy classes of morphisms $\op{Q}_{2n-1}\to G$; this bijection is, by \cite[Theorem 4.3.6]{AsokDoranFasel}, induced by mapping a clutching function $\op{Q}_{2n-1}\to G$ to the $G$-torsor associated to the composition $\op{D}(Z^2-1)\to \op{Q}_{2n-1}\to G$. Composing further with the restriction of the clutching function along $\op{Q}_{2n-1}\to\op{D}(Z^2-1)$ induces a map
\[
[\op{Q}_{2n-1},G]_{\mathbb{A}^1}\to \op{H}^1_{\op{Nis}}(\op{Q}_{2n},G) \to [\op{Q}_{2n-1},G]_{\mathbb{A}^1}.
\]
The first map is the natural bijection between clutching functions and torsors, and the composition is the identity. In particular, if we have a trivialization of a $G$-torsor over the open subschemes $\op{D}(Z\pm 1)$ with the associated transition function $\op{D}(Z^2-1)\to G$, then the composition $\op{Q}_{2n-1}\to\op{D}(Z^2-1)\to G$ is a clutching function for the given $G$-torsor. 
\end{proof}

\begin{remark}
Via the natural bijections in the proof of Proposition~\ref{prop:clutching}, we also see that every rationally trivial $G$-torsor (for an isotropic group $G$) has a trivialization in the covering $\op{D}(Z\pm 1)$ of $\op{Q}_{2n}$, with a transition function which is pulled back along $\op{D}(Z^2-1)\to \op{Q}_{2n-1}$.
\end{remark}

\subsection{Trivialization and characteristic map} 

To identify the characteristic map $\op{Q}_{2n}\to {\op{B}}_{\op{Nis}}\op{SO}(2n)$, classifying the $\op{SO}(2n)$-torsor $\pi\colon \op{SO}(2n+1)\to\op{Q}_{2n}$, we proceed as in \cite[Section 23]{steenrod}. The first step is to write down explicit trivializations of the torsor $\pi$ over the two principal open subsets $\op{D}(1+Z)$ and $\op{D}(1-Z)$; this provides an explicit transition map $\op{D}(Z^2-1)\to\op{SO}(2n)$. As a second step, the clutching function is obtained as the composition $\op{Q}_{2n-1}\xrightarrow{\iota}\op{D}(Z^2-1)\to \op{SO}(2n)$. The characteristic map $\op{Q}_{2n}\to {\op{B}}_{\op{Nis}}\op{SO}(2n)$ then corresponds to the clutching map $\op{Q}_{2n-1}\to\op{SO}(2n)$ via the clutching result \cite[Theorem 4.3.6]{AsokDoranFasel}, cf. also the recollection as above. 

To obtain an explicit trivialization of the torsor $\pi$ over the principal open subset $\op{D}(1+Z)$, we follow the construction in \cite[23.3]{steenrod}: in classical topology, one fixes a base point $x_n\in \op{S}^n$ and sends a vector $x$ to the matrix which is the rotation from $x_n$ to $x$ on the plane spanned by $x_n$ and $x$, and is the identity on the orthogonal complement of that plane. 

We want to employ the same construction for the split orthogonal groups. For a coordinate-free formula for the above rotation, we can use the composition of two reflections at hyperplanes, one perpendicular to $x$ and the other perpendicular to the bisector of the angle between $x$ and $x_n$ for which we can take $x+x_n$ since both are unit vectors.\footnote{This is discussed in amd's answer to Math.StackExchange question 1909717 ``Rotation matrix in terms of dot products''.} The relevant coordinate-free expression for the rotation is then 
\[
a\mapsto a-\frac{B(x_n+x,a)}{1+B(x_n, x)}(x_n+x)+2B(x_n, a)x,
\]
which in the situation of \cite[23.3]{steenrod} reproduces exactly the matrix in loc.cit. 

Now we consider this formula in the odd-dimensionsal split setting, given by the bilinear form $B(u,v)=u^{\op{t}}Jv$. Recall that we denote vectors in $V=F^{2n+1}$ by $(t_1,\dots,t_n,t_0,t_{-n},\dots,t_{-1})$, with natural basis  given by $(\op{e}_1,\dots,\op{e}_n,\op{e}_0,\op{e}_{-n},\dots,\op{e}_{-1})$. The corresponding orthogonal basis (with vectors of length $1$ or $-1$) is $v_{\pm i}=\op{e}_i\pm \op{e}_{-i}$ and $v_0=\op{e}_0$. We choose $v_0$ as the relevant base point which should be rotated into a given vector $u=(t_1,\dots,t_n,t_0,t_{-n},\dots,t_{-1})$ on $\op{Q}_{2n}$. 
The rotation applied to $v_{\pm i}$ yields
\[
v_{\pm i}- \frac{B(u+v_0,v_{\pm i})}{1+B(u,v_0)}(u+v_0)+2B(v_0,v_{\pm i})u=v_{\pm i}-\frac{\left(t_{-1}\pm t_1\right)}{1+t_0}(u+v_0),
\]
and the rotation applied to $v_0$ yields
\[
v_0-\frac{B(u+v_0,v_0)}{1+B(u,v_0)}(u+v_0)+2B(v_0,v_0)u=u.
\]
Now we express $u=(t_1,\dots,t_n,t_0,t_{-n},\dots,t_{-1})$ in the above orthogonal basis:
\[
u=t_0v_0+\sum_{i=1}^n\left(\left(\frac{t_i+t_{-i}}{2}\right)v_i+ \left(\frac{t_i-t_{-i}}{2}\right)v_{-i}\right) 
\]
In particular, the entries of the representing matrix for the rotation above are given as follows: the $v_{\pm j}$-coordinate of the image of $v_{\pm i}$ is given as
\[
\delta_{ij}-\frac{\left(t_{-i}\pm_i t_i\right)\left(t_j\pm_j t_{-j}\right)}{2+2t_0},
\]
the $v_0$-coordinate of the image of $v_{\pm j}$ is given by $(t_{-j}\pm t_j)$, and the coordinates of the image of $v_0=u$ are as above.\footnote{The representing matrix for the relevant rotation has a form very similar to the rotation matrices appearing in \cite[Section 23.3]{steenrod}, replacing the coordinates $t_i$ by $t_i\pm t_{-i}$ and the denominator $1+t_n$ by $2+2t_0$.} The fact that the matrix is orthogonal and maps $v_0$ to $u$ (provided that $1+t_0$ is invertible) follows from the construction. We denote by $\rho\colon \op{D}(1+Z)\to \op{SO}(2n+1)$ the corresponding  map, taking $u$ to the rotation described (mapping $v_0$ to $u$) above. 

As a special case, for $u=v_1$, we get the matrix
\[
(\phi(v_1))_{ij}=\left\{\begin{array}{rl}
1 & i=j \textrm{ or } i=0,j=1\\
-1 & i=1,j=0\\
0 &\textrm{otherwise}
\end{array}\right.
\]
and $\lambda=\phi(v_1)^2$ is the diagonal matrix whose entries are all 1 except $-1$ for the entries corresponding to $v_0$ and $v_1$, representing the $180^\circ$ degree rotation in the plane spanned by $v_0$ and $v_1$. 

Then we can directly follow the definition in \cite[23.3]{steenrod}: the trivializing maps  are given by
\begin{eqnarray*}
\phi_1\colon \op{D}(1+Z)\times \op{SO}(2n)&\to& \op{SO}(2n+1)\colon  (u,r)\mapsto \rho(u)\cdot r\\
\phi_2\colon \op{D}(1-Z)\times \op{SO}(2n) & \to & \op{SO}(2n+1)\colon  (u,r)\mapsto \lambda\phi(\lambda(u))\cdot r
\end{eqnarray*}
The corresponding partial inverses $\op{SO}(2n+1)\dashrightarrow\op{D}(1\pm Z)\times\op{SO}(2n)$ can be obtained as follows: the image in the sphere is simply given by the projection $\pi\colon \op{SO}(2n+1)\to\op{Q}_{2n}\colon A\mapsto Av_0$ (and obviously this is only defined whenever the image lands in the respective coordinate neighbourhood $\op{D}(1\pm Z)$). The appropriate rotation in $\op{SO}(2n)$ can be recovered via
\[
p_1(r)=\phi(\pi(r))^{-1}r \quad \textrm{ and } \quad p_2(r)=\phi(\lambda \pi(r))^{-1}\lambda r.
\]

On the intersection of the coordinate neighbourhoods $\op{D}(Z^2-1)$, the transition function is given as 
\[
\tau\colon \op{D}(Z^2-1)\to \op{SO}(2n+1)\colon x\mapsto \tau(x)=\phi(x)^{-1}\cdot\lambda\phi(\lambda(x)).
\]
In particular, since $\phi(v_1)$ and $\lambda\phi(\lambda(v_1))$ are both the rotation from $v_0$ to $v_1$, we find that $\tau(v_1)$ is the identity, making $\tau$ a pointed morphism. 

As in \cite[§23]{steenrod}, it can be checked that the transition function maps a point $u\in \op{Q}_{2n-1}=\op{V}(Z)\subset \op{D}(Z^2-1)$ to the composition of a reflection for the vector $v_1$ followed by a reflection for the vector $u$. Since $v_1$ is a unit vector, the reflection for $v_1$ is given by $x\mapsto x-B(v_1,x)v_1$. For the orthogonal basis $v_{\pm i}$, every vector except $v_1$ is fixed, and $v_1$ changes sign. Then the reflection for $u$ maps $v_{\pm i}\mapsto v_{\pm i}-2B(u,v_{\pm i})u$ (again using that $u$ is a unit vector). If we write, as above, 
\[
u=\sum_{i=1}^n\left(\left(\frac{t_i+t_{-i}}{2}\right)v_i+ \left(\frac{t_i-t_{-i}}{2}\right)v_{-i}\right) 
\]
we get $v_{1}\mapsto -v_{1}+2(t_1+ t_{-1})u$ and $v_{\pm i}\mapsto v_{\pm i}-2(t_i\pm t_{-i})u$ for everything else. Consequently, the matrix for the transition function, expressed in the basis $(v_1,\dots,v_n,v_{-n},\dots,v_{-1})$ is given by 
\[
\tau_{ij}=(\delta_{ij}-2(t_i\pm_i t_{-i})(t_j\pm_j t_{-j}))_{ij}\cdot\op{diag}(-1,1,\dots,1).
\]

We can now compose this map $\op{Q}_{2n-1}\to \op{SO}(2n)$ with the natural projection $\op{SO}(2n)\to\op{Q}_{2n-1}\colon A\mapsto A\cdot v_1$. The resulting morphism $\pi\circ \tau\colon \op{Q}_{2n-1}\to\op{Q}_{2n-1}$ is given by
\[
 u=(t_1,\dots,t_n,t_{-n},\dots,t_{-1})\mapsto -v_1+(t_1+t_{-1})\sum_{i=1}^n\left((t_i+t_{-i})v_i+(t_i-t_{-i})v_{-i}\right)
\]
which can equivalently be written as
\[
((t_1+t_{-1})t_1-1,\dots,(t_1+t_{-1})t_n,(t_1+t_{-1})t_{-n},\dots,(t_1+t_{-1})t_{-1}-1). 
\]
This is now the morphism $\op{Q}_{2n-1}\to\op{Q}_{2n-1}$ which induces the connecting map in the long exact $\mathbb{A}^1$-homotopy sequence.

\subsection{A degree computation}
\label{sec:degree}

Now we need to compute the degree of the morphism $\op{Q}_{2n-1}\to\op{Q}_{2n-1}$ determined above. We will see later that the attaching map for the long exact $\mathbb{A}^1$-homotopy sequence for $\op{Q}_{2n-1}\to\op{V}_{2,2n+1}\to\op{Q}_{2n}$ is then obtained as the composition $\op{Q}_{2n-1}\to\op{SO}(2n)\to\op{Q}_{2n-1}$ of the clutching map $\op{Q}_{2n-1}\to\op{SO}(2n)$ with the natural projection $\op{SO}(2n)\to\op{Q}_{2n-1}$. Hence the degree of the above morphism $\op{Q}_{2n-1}\to\op{Q}_{2n-1}$ completely describes the attaching map in the long exact homotopy sequence.

By \cite[Proposition 4.1, Corollary 4.5]{AsokFaselSpheres}, the natural map
\[
[\mathbb{A}^n\setminus\{0\},\mathbb{A}^n\setminus\{0\}]_{\mathbb{A}^1}\to \op{H}^{n-1}(\mathbb{A}^n\setminus\{0\},\mathbf{K}^{\op{MW}}_n)
\]
induced by the identification $\tau_{\leq n-1}\mathbb{A}^n\setminus\{0\}\cong \op{K}(\mathbf{K}^{\op{MW}}_n,n-1)$ is a bijection. Via the natural $\mathbb{A}^1$-equivalence $\op{Q}_{2n-1}\to\mathbb{A}^n\setminus\{0\}$, we also get a natural bijection $[\op{Q}_{2n-1},\op{Q}_{2n-1}]_{\mathbb{A}^1}\cong \op{H}^{n-1}(\op{Q}_{2n-1},\mathbf{K}^{\op{MW}}_n)$. 

Therefore, the degree of a morphism $\op{Q}_{2n-1}\to\op{Q}_{2n-1}$ can be determined via cohomology. First, a generator for $\op{H}^{n-1}(\op{Q}_{2n-1},\mathbf{K}^{\op{MW}}_n)\cong \op{GW}(F)$ is given by the orientation class, cf. \cite[Example 4.4]{AsokFaselSpheres}: writing $\op{Q}_{2n-1}=\op{V}(\sum_{i=1}^nX_iX_{-i}-1)$, the vanishing locus of the $X_2,\dots,X_n$ defines a smooth scheme $Y\subset \op{Q}_{2n-1}$ of codimension $n-1$. Then $X_1$ is an invertible function on $Y$, hence gives rise to $[X_1]\in \mathbf{K}^{\op{MW}}_1(F(Y))$. The Koszul complex for the regular sequence $X_2,\dots,X_n$ gives a generator of the $F(Y)$-vector space $\omega_Y=\op{Ext}^{n-1}_{\mathscr{O}_X,Y}(F(Y),\mathscr{O}_{X,Y})$ and an isomorphism $\mathbf{K}^{\op{MW}}_1(F(Y))\cong\mathbf{K}^{\op{MW}}_1(F(Y),\omega_Y)$. Then the class $[X_1]\in\mathbf{K}^{\op{MW}}_1(F(Y),\omega_Y)$ is a cocycle in $\op{H}^{n-1}(\op{Q}_{2n-1},\mathbf{K}^{\op{MW}}_n)$, which is the orientation class $\xi$. With this definition, the degree of the morphism  $f\colon \op{Q}_{2n-1}\to\op{Q}_{2n-1}$ can now be determined via $\deg f\cdot\xi=f^\ast(x)$.

\begin{proposition}
\label{prop:degree}
For the morphism $\pi\circ\tau\colon \op{Q}_{2n-1}\to\op{Q}_{2n-1}$ given by mapping $(t_1,\dots,t_n,t_{-n},\dots,t_{-1})$ to  
\[
((t_1+t_{-1})t_1-1, \dots, (t_1+t_{-1})t_n, (t_1+t_{-1})t_{-n}, \dots,(t_1+t_{-1})t_{-1}-1),
\]
we have 
\[
\deg(\pi\circ \tau)=\left\{\begin{array}{ll}
2&n \equiv 0\bmod 2\\
\mathbb{H}&n\equiv 1\bmod 2
\end{array}\right.
\]
\end{proposition}

\begin{proof}
We compute the pullback of the orientation class $\xi$ along $\pi\circ\tau$. Note that the above map is not flat. This can be seen e.g. by the fact that the preimage of the subscheme $\op{V}(X_2,\dots,X_n)$ under $\pi\circ\tau$ is 
\[
\op{V}((X_1+X_{-1})X_2,\dots,(X_1+X_{-1})X_n)= \op{V}(X_1+X_{-1})\cup\op{V}(X_2,\dots,X_n)
\]
and thus has two components of different (co)dimension. To compute the pullback in terms of Gersten complexes, we proceed as in \cite{rost}: factor the morphism $\pi\circ\tau$ as
\[
\op{Q}_{2n-1}\xrightarrow{(\op{id},\pi\circ\tau)}\op{Q}_{2n-1}\times\op{Q}_{2n-1} \xrightarrow{\op{pr}_2}\op{Q}_{2n-1}
\]
of a regular embedding (whose normal bundle is the pullback $(\pi\circ\tau)^\ast\mathscr{T}\op{Q}_{2n-1}$ of the tangent bundle of $\op{Q}_{2n-1}$) and a smooth morphism. 

The pullback of the class on $\op{Q}_{2n-1}$ along $\op{pr}_2$ can be described easily: it is given by the class of $[X_1]$ on the subscheme $\op{V}(X_2,\dots,X_n)$, where here the $X_i$ are viewed as functions on $\op{Q}_{2n-1}\times\op{Q}_{2n-1}$ via the projection $\op{pr}_2$. 

For the pullback along the regular embedding $f:Y\hookrightarrow X$, we can use the description given in \cite{rost} or \cite{fasel:chowwittring}, cf. \cite{AsokFaselEuler} for an identification with pullbacks in sheaf cohomology. This definition uses the deformation to the normal cone: in the deformation space $D(X,Y)=\op{Bl}_{Y\times \{0\}}X\times\mathbb{A}^1\setminus \op{Bl}_YX$, we have the normal bundle $\mathscr{N}_YX\hookrightarrow D(X,Y)$ embedded as divisor, with open complement $X\times\mathbb{A}^1\setminus\{0\}$. The pullback $\op{H}^i(X,\mathbf{K}^{\op{MW}}_j)\to \op{H}^i(Y,\mathbf{K}^{\op{MW}}_j)$ is then defined by composing the flat pullback along $X\times\mathbb{A}^1\setminus\{0\}$, multiplication by the uniformizer of $X\times\{0\}\subseteq X\times\mathbb{A}^1$, followed by the boundary map to the divisor $\mathscr{N}_YX$ and finally the homotopy invariance map to $Y$. Tracing through the definition in the specific case of the graph embedding $\op{Q}_{2n-1}\hookrightarrow\op{Q}_{2n-1}\times\op{Q}_{2n-1}$, we see that the component $\op{V}(X_1+X_{-1})$ doesn't contribute to the pullback cycle (essentially for codimension reasons). The underlying scheme of the pullback cycle is the other component of the preimage, $\op{V}(X_2,\dots,X_n)$, and the pullback of the class $[X_1]\in\mathbf{K}^{\op{MW}}_1(F(Y),\omega_Y)$ is given by 
\[
[(X_1+X_{-1})X_1-1]\in \mathbf{K}^{\op{MW}}_1(F(Y),\omega_Y).
\]
But on $F(Y)$, we have $X_1X_{-1}=1$, hence $[(X_1+X_{-1})X_1-1]=[X_1^2]$. The degree $\deg(\pi\circ\tau)$ is therefore determined by the equality $[X_1^2]=\deg(\pi\circ\tau)[X_1]$ in $\mathbf{K}^{\op{MW}}_1(F(Y))$. 

If $n$ is even, we have $[X_1^2]=2[X_1]$, as discussed in the proof of \cite[Theorem 4.9]{AsokFaselSpheres}. If $n$ is odd, we have $[X_1^2]=\mathbb{H}[X_1]$, as discussed in the proof of \cite[Theorem 4.11]{AsokFaselSpheres}. This proves the claim.
\end{proof}

\begin{remark}
Note that this fits the classical pattern. From \cite[Lemma 1.7]{levine}, we knnow that the reduced motivic Euler characteristic of the sphere $\op{S}^{p,q}=(\op{S}^1_s)^{\wedge p-q}\wedge (\mathbb{G}_{\op{m}})^{\wedge q}$ is $\chi(\op{S}^{p,q})=(-1)^p\langle-1\rangle^q\in \op{GW}(F)$. This is the inverse of the $\mathbb{A}^1$-Brouwer degree of the antipodal map $\op{S}^{p,q}\to\op{S}^{p,q}$. In the particular case $\op{Q}_{2n-1}\simeq \op{S}^{2n-1,n}$, the Euler characteristic is $-\langle-1\rangle ^n$ and the degree of the antipodal map is $\langle-1\rangle^n$. The degree of the map $\op{Q}_{2n-1}\to\op{Q}_{2n-1}$ considered above is $1+\langle-1\rangle^n$, i.e. 1 plus the degree of the antipodal map, recovering exactly the classical result in \cite[Section 23]{steenrod}. 
\end{remark}

\begin{remark}
For $F=\mathbb{C}$, the complex realization, the motivic degree of the map in $\op{GW}(\mathbb{C})\cong\mathbb{Z}$ is always 2, independent of $n$. This fits exactly with the classical situation of the torsor $\op{SO}(2n+1)\to \op{S}^{2n}$, where the degree of the map $\op{S}^{2n-1}\to\op{S}^{2n-1}$ will always be $2$ because the sphere is always odd-dimensional.\footnote{This is one of the few cases in this paper where $\op{SO}(2n+1)$ denotes the compact form, or alternatively the split group $\op{SO}(2n+1,\mathbb{C})$.}

For $F=\mathbb{R}$, the $\mathbb{A}^1$-Brouwer degree of the map in $\op{GW}(\mathbb{R})\cong\mathbb{Z}$ is a rank 2 form whose signature is $2$ or $0$, depending if $n$ is even or odd. This fits exactly with the topological situation: the split form of $\op{O}(2n+1)$ has maximal compact subgroup $\op{O}(n)\times\op{O}(n+1)$ and the relevant torsor $\op{O}(n)$-torsor is $\op{O}(n+1)\to\op{S}^n$. The degree of the attaching map $\op{S}^n\to\op{S}^n$ is $2$ or $0$, depending if $n$ is even or odd, exactly as in the motivic situation.
\end{remark}

\section{On \texorpdfstring{$\mathbb{A}^1$}{A1}-homotopy groups of orthogonal Stiefel varieties}
\label{sec:stiefel}

Now we have enough background on the geometry of orthogonal Stiefel varieties $\op{V}_{2,n}$ to determine their $\mathbb{A}^1$-connectivity and the first non-vanishing $\mathbb{A}^1$-homotopy sheaves. The Nisnevich cohomology groups with coefficients in the first non-vanishing $\mathbb{A}^1$-homotopy sheaf will be the natural home for obstructions to splitting off hyperbolic planes. 

\begin{proposition}
\label{prop:stiefel1}
The Stiefel variety $\op{V}_{2,2d+1}$ is $\mathbb{A}^1$-$(d-2)$-connected. The first non-vanishing $\mathbb{A}^1$-homotopy sheaf is
\[
\bm{\pi}^{\mathbb{A}^1}_{d-1}(\op{V}_{2,2d+1})\cong \mathbf{K}^{\op{MW}}_d/(1+\langle(-1)^d\rangle)\cong \left\{\begin{array}{ll}
\mathbf{K}^{\op{MW}}_d/2 & d\equiv 0\mod 2\\
\mathbf{K}^{\op{MW}}_d/[\mathbb{H}]\cong\mathbf{I}^d & d\equiv 1\mod 2\end{array}\right.
\]
\end{proposition}

\begin{proof}
The result follows from the long exact sequence in $\mathbb{A}^1$-homotopy associated to the fibration 
\[
\op{Q}_{2d-1}\to \op{V}_{2,2d+1}\to \op{Q}_{2d}
\]
from Lemma~\ref{lem:quadfib}. The relevant portion of the long exact $\mathbb{A}^1$-homotopy sequence is the following
\[
\bm{\pi}_{d}^{\mathbb{A}^1}(\op{Q}_{2d})\stackrel{\partial}{\longrightarrow} \bm{\pi}_{d-1}^{\mathbb{A}^1}(\op{Q}_{2d-1})\to\bm{\pi}_{d-1}^{\mathbb{A}^1}(\op{V}_{2,2d+1})\to \bm{\pi}_{d-1}^{\mathbb{A}^1}(\op{Q}_{2d}).
\]
By Morel's computations \cite{MField}, $\op{Q}_{2d}$ is $\mathbb{A}^1$-$(d-1)$-connected and $\op{Q}_{2d-1}$ is $\mathbb{A}^1$-$(d-2)$-connected. In particular, all $\mathbb{A}^1$-homotopy sheaves to the right of $\bm{\pi}_{d-1}^{\mathbb{A}^1}(\op{V}_{2,2d+1})$ will vanish. Moreover, Morel's computations \cite{MField} imply that $\bm{\pi}^{\mathbb{A}^1}_d(\op{Q}_{2d})\cong\mathbf{K}^{\op{MW}}_d$ and $\bm{\pi}^{\mathbb{A}^1}_d(\op{Q}_{2d-1})\cong\mathbf{K}^{\op{MW}}_d$. In particular, we can rewrite the above exact sequence  as
\[
\bm{\pi}_{d-1}^{\mathbb{A}^1}(\op{V}_{2,2d+1})\cong\op{coker}\left( \mathbf{K}^{\op{MW}}_d \stackrel{\partial}{\longrightarrow} \mathbf{K}^{\op{MW}}_d\right).
\]

It remains to identify the map $\partial$. In the category of strictly $\mathbb{A}^1$-invariant sheaves of abelian groups, we have an identification
\[
\op{Hom}(\mathbf{K}^{\op{MW}}_d,\mathbf{K}^{\op{MW}}_d)\cong\op{K}^{\op{MW}}_0(F), 
\]
so that $\partial$ is given by multiplication by some element $\partial\in\op{GW}(F)$. Moreover, since $[\op{Q}_{2d},\op{Q}_{2d}]_{\mathbb{A}^1}\cong [\op{Q}_{2d-1},\op{Q}_{2d-1}]_{\mathbb{A}^1}\cong\op{GW}(F)$, we can identify $\partial\in \op{GW}(F)$ by tracing through the definition of the connecting map for the special case $\op{id}_{\op{Q}_{2d}}\in \pi^{\mathbb{A}^1}_d(\op{Q}_{2d})(\op{Q}_{2d})$. The result will be a map $\op{Q}_{2d-1}\to\op{Q}_{2d-1}$ which represents $\partial\in\op{GW}(F)$. By Lemma~\ref{lem:quadfib}, the connecting map for the fiber sequence $\op{Q}_{2d-1}\to \op{V}_{2,2d+1}\to\op{Q}_{2d}$ can be identified with the connecting map for the fiber sequence $\op{Q}_{2d-1}\to {\op{B}}_{\op{Nis}}\op{SO}(2d-1)\to {\op{B}}_{\op{Nis}}\op{SO}(2d)$, applied to the map $\op{Q}_{2d}\to{\op{B}}_{\op{Nis}}\op{SO}(2d)$ which classifies the $\op{SO}(2d)$-torsor $\op{SO}(2d+1)\to\op{SO}(2d+1)/\op{SO}(2d)\cong \op{Q}_{2d}$. 

To identify the connecting map for the fiber sequence $\op{Q}_{2d-1}\to{\op{B}}_{\op{Nis}}\op{SO}(2d-1)\to {\op{B}}_{\op{Nis}}\op{SO}(2d)$, recall that the classifying map $\op{Q}_{2d}\to{\op{B}}_{\op{Nis}}\op{SO}(2d)$ corresponds, via the clutching construction, to a morphism $\op{Q}_{2d-1}\to\op{SO}(2d)$. Extending the fiber sequence one step to the left, we get a fiber sequence $\op{SO}(2d)\to \op{Q}_{2d-1}\to{\op{B}}_{\op{Nis}}\op{SO}(2d-1)$. Now the connecting map takes $\op{Q}_{2d}\to{\op{B}}_{\op{Nis}}\op{SO}(2d)$ to the composition
\[
\op{Q}_{2d-1}\to\op{SO}(2d)\to \op{Q}_{2d-1},
\]
where the first map is the clutching map and the second map is the natural projection. By the previous discussion, the element $\partial\in\op{GW}(F)$ is the degree of this composition. The claim then follows from Proposition~\ref{prop:degree}. 

It remains to identify the cohomology of $\mathbf{K}^{\op{MW}}_d/[\mathbb{H}]$. This is the Witt K-theory $\mathbf{K}^{\op{W}}_\bullet$, cf. \cite[p.61]{MField}. The identification $\mathbf{K}^{\op{W}}_d\cong \mathbf{I}^d$  follows from \cite[Theorem 2.1]{morel:puissances}
\end{proof}

\begin{remark}
In the complex realization, the above is the case $\op{V}_{2,2d+1}$ is the one with the fibration $\op{S}^{2d-1}\to\op{V}_{2,2d+1}\to\op{S}^{2d}$ where the boundary map $\pi_{i+1}(\op{S}^{2d-1})\to\pi_{i}(\op{S}^{2d-2})$ is multiplication by $2$. 

Proposition~\ref{prop:stiefel1} implies that for any field $F$ of characteristic $\neq 2$, we have 
\[
[\op{Q}_{2d-2},\op{V}_{2,2d+1}]_{\mathbb{A}^1}\cong \left\{\begin{array}{ll}
\op{K}^{\op{MW}}_1(F^\times)/2 & d\equiv 0\mod 2\\
\op{I}(F) & d\equiv 1\mod 2\end{array}\right.
\]
which are both trivial over algebraically closed fields, consistent with the connectivity of the Stiefel varieties over $\mathbb{C}$.
On the other hand,  we have 
\[
[\op{Q}_{2d-1},\op{V}_{2,2d+1}]_{\mathbb{A}^1}\cong \left\{\begin{array}{ll}
\op{GW}(F)/2 & d\equiv 0\mod 2\\
\op{W}(F) & d\equiv 1\mod 2\end{array}\right.
\]
Over algebraically closed fields, both cases are isomorphic to $\mathbb{Z}/2\mathbb{Z}$ recovering the classical isomorphism $\pi_{2d-1}\op{V}_{2,2d+1}(\mathbb{C})\cong\mathbb{Z}/2\mathbb{Z}$.

In real realization, there are two cases for the homotopy groups of $\op{V}_{2,2d+1}(\mathbb{R})\cong\op{SO}(d+1,d)/\op{SO}(d,d-1)$. The result is $\mathbb{Z}/2\mathbb{Z}$ whenever $d$ is even and $\mathbb{Z}$ whenever $d$ is odd. This is exactly what the above formula reproduces for the real realization.
\end{remark}

Combining the above technique with Morel's Freudenthal suspension theorem \cite[Theorem 5.60]{MField}, some more $\mathbb{A}^1$-homotopy sheaves of Stiefel varieties can be described. The description is not as explicit as the one above, due to our limited knowledge of unstable $\mathbb{A}^1$-homotopy sheaves of spheres. 

\begin{theorem}
\label{thm:stiefelhigh}
Let $F$ be a perfect field of characteristic $\neq 2$, and let $d\geq 3$. Then for all $i\leq 2(d-2)$, there are exact sequences  
\[
0\to \bm{\pi}_i^{\mathbb{A}^1}(\op{Q}_{2d-1})/a \to \bm{\pi}^{\mathbb{A}^1}_i(\op{V}_{2,2d+1})\to \ker\left(a\colon\bm{\pi}^{\mathbb{A}^1}_{i-1}(\op{Q}_{2d-1})\to \bm{\pi}^{\mathbb{A}^1}_{i-1}(\op{Q}_{2d-1})\right)\to 0
\]
where $a=2$ if $d\equiv 0\bmod 2$ and $a=[\mathbb{H}]$ if $d\equiv 1\bmod 2$. 
\end{theorem}

\begin{proof}
The connecting map can be reinterpreted as being induced from the morphism $\Omega\op{Q}_{2d}\to\op{Q}_{2d-1}$. Then we can consider the composition
\[
\op{Q}_{2d-1}\to\Omega_{\mathbb{A}^1}\Sigma_{\mathbb{A}^1}\op{Q}_{2d-1}\to\op{Q}_{2d-1}.
\]
The first map is the unit of the adjunction $\Omega_{\mathbb{A}^1}\dashv\Sigma_{\mathbb{A}^1}$, the second map is the morphism inducing the connecting map. Morel's Freudenthal suspension theorem \cite[Theorem 5.60]{MField} implies that the first map is $\mathbb{A}^1$-$2(d-2)$-connected since the sphere $\op{Q}_{2d-1}$ is $\mathbb{A}^1$-$(d-2)$-connected. In particular, the connecting map for $\bm{\pi}^{\mathbb{A}^1}_i$ with $i\leq 2(d-2)$ can be identified with multiplication by the degree of the composition $\op{Q}_{2d-1}\to\Omega\Sigma\op{Q}_{2d-1}\to\op{Q}_{2d-1}$ above. Note that the first map is multiplication by some unit, but that doesn't affect the computation of the kernel or cokernel. The claim then follows from the long exact homotopy sequence of Lemma~\ref{lem:quadfib}.
\end{proof}

\begin{remark}
At this point, not even the $\mathbb{A}^1$-homotopy sheaves $\bm{\pi}^{\mathbb{A}^1}_n(\op{Q}_{2n-1})$ are known. There is an explicit conjecture, though, formulated in \cite[Conjecture 7]{AsokFaselMetastable}. It is expected that for $n\geq 4$ there is an exact sequence of strictly $\mathbb{A}^1$-invariant sheaves of abelian groups of the form
\[
\mathbf{K}^{\op{M}}_{n+2}/24\to \bm{\pi}^{\mathbb{A}^1}_n(\op{Q}_{2n-1})\to \mathbf{GW}^n_{n+1}
\] 
which becomes short exact after $n$-fold contraction.
\end{remark}

\begin{proposition}
\label{prop:stiefel2}
The Stiefel variety $\op{V}_{2,2d}$ is $\mathbb{A}^1$-$(d-2)$-connected. For any $n$, we have 
\[
\bm{\pi}^{\mathbb{A}^1}_n(\op{V}_{2,2d})\cong \bm{\pi}^{\mathbb{A}^1}_n(\op{Q}_{2d-2})\times \bm{\pi}^{\mathbb{A}^1}_{n}(\op{Q}_{2d-1}).
\]
In particular, the first non-vanishing $\mathbb{A}^1$-homotopy sheaf is
\[
\bm{\pi}^{\mathbb{A}^1}_{d-1}(\op{V}_{2,2d})\cong \mathbf{K}^{\op{MW}}_{d-1}\times\mathbf{K}^{\op{MW}}_d.
\]
\end{proposition}

\begin{proof}
The result follows from the long exact sequence in $\mathbb{A}^1$-homotopy associated to the fibration 
\[
\op{Q}_{2d-2}\to \op{V}_{2,2d}\to \op{Q}_{2d-1}
\]
from Lemma~\ref{lem:quadfib}. The relevant portion of the long exact $\mathbb{A}^1$-homotopy sequence is the following
\[
\bm{\pi}_{d}^{\mathbb{A}^1}(\op{Q}_{2d-1})\stackrel{\partial}{\longrightarrow} \bm{\pi}_{d-1}^{\mathbb{A}^1}(\op{Q}_{2d-2})\to\bm{\pi}_{d-1}^{\mathbb{A}^1}(\op{V}_{2,2d})\to \bm{\pi}_{d-1}^{\mathbb{A}^1}(\op{Q}_{2d-1})\to \bm{\pi}_{d-2}^{\mathbb{A}^1}(\op{Q}_{2d-2}).
\]
Since $\op{Q}_{2d-1}$ and $\op{Q}_{2d-2}$ are both $\mathbb{A}^1$-$(d-2)$-connected all $\mathbb{A}^1$-homotopy sheaves to the right of $\bm{\pi}_{d-1}^{\mathbb{A}^1}(\op{Q}_{2d-1})$ will vanish. Using the information on $\mathbb{A}^1$-homotopy groups of spheres, cf. Proposition~\ref{prop:stiefel1}, we can rewrite the above exact sequence as
\[
\bm{\pi}_{d}^{\mathbb{A}^1}(\mathbb{A}^d\setminus\{0\})\stackrel{\partial}{\longrightarrow} \mathbf{K}^{\op{MW}}_{d-1}\to\bm{\pi}_{d-1}^{\mathbb{A}^1}(\op{V}_{2,2d})\to \mathbf{K}^{\op{MW}}_d\to 0
\]
We claim that the projection $\op{V}_{2,2d}\to\op{Q}_{2d-1}$ has a section. This follows from relative obstruction theory as discussed in Section~\ref{sec:prelims}. There is a sequence of obstructions for lifting the identity on $\op{Q}_{2d-1}$ to a map $\op{Q}_{2d-1}\to\op{V}_{2,2d}$, these obstructions live in cohomology groups $\op{H}^{i+1}_{\op{Nis}}(\op{Q}_{2d-1},\bm{\pi}^{\mathbb{A}^1}_i(\op{Q}_{2d-2}))$ of the base with coefficients in $\mathbb{A}^1$-homotopy sheaves of the fiber. The only possibly non-trivial cohomology group of $\op{Q}_{2d-1}$ is 
\[
\op{H}^{d-1}_{\op{Nis}}(\op{Q}_{2d-1},\bm{\pi}^{\mathbb{A}^1}_{d-2}(\op{Q}_{2d-2}))\cong \bm{\pi}^{\mathbb{A}^1}_{d-2}(\op{Q}_{2d-2})_{-d}. 
\]
Since the sphere $\op{Q}_{2d-2}$ is $\mathbb{A}^1$-$(d-2)$-connected, this implies that all the obstruction groups for lifting the identity vanish, and we obtain the required section $\op{Q}_{2d-1}\to\op{V}_{2,2d}$. As a consequence, the boundary map $\partial$ in the above $\mathbb{A}^1$-homotopy sequence is trivial, and the extension is split. This proves all the claims.
\end{proof}

\begin{remark}
Under complex realization, the case $\op{V}_{2,2d}$ is the one with the fibration $\op{S}^{2d-2}\to\op{V}_{2,2d}\to\op{S}^{2d-1}$. All the boundary map $\pi_{i+1}(\op{S}^{2d-1})\to\pi_{i}(\op{S}^{2d-2})$ are zero because they are induced by the sum of the identity and the antipodal map on the $2d-2$-sphere. Therefore, we get short exact sequences 
\[
0\to \pi_n(\op{S}^{2d-2})\to \pi_n(\op{V}_{2,2d})\to\pi_n(\op{S}^{2d-1})\to 0.
\]
The first non-vanishing homotopy is then $\pi_{2d-2}(\op{V}_{2,2d}(\mathbb{C}))\cong\mathbb{Z}$; its motivic analogue is 
\[
[\op{Q}_{2d-2},\op{V}_{2,2d}]_{\mathbb{A}^1}\cong \left(\mathbf{K}^{\op{MW}}_{d-1}\right)_{-(d-1)}\times \left(\mathbf{K}^{\op{MW}}_{d}\right)_{-(d-1)}\cong \op{GW}(k)\times \op{K}^{\op{MW}}_1(k).
\]
Over an algebraically closed field, the first summand corresponds to the classical homotopy group and the second summand is $\op{K}^{\op{M}}_1(k)\cong k^\times$; modulo primes only the first summand is visible. The next homotopy group is $\pi_{2d-1}(\op{V}_{2,2d}(\mathbb{C}))\cong \mathbb{Z}/2\mathbb{Z}\times\mathbb{Z}$ with the following motivic analogue 
\[
[\op{Q}_{2d-1},\op{V}_{2,2d}]_{\mathbb{A}^1}\cong \left(\mathbf{K}^{\op{MW}}_{d-1}\right)_{-d}\times \left(\mathbf{K}^{\op{MW}}_{d}\right)_{-d}\cong \op{W}(k)\times\op{GW}(k).
\]

In the real realization, Proposition~\ref{prop:stiefel2} corresponds to the case where 
\[
\op{SO}(d,d)/\op{SO}(d-1,d-1)\cong \op{S}^{d-1}\times\op{S}^{d-1}. 
\]
\end{remark}

\begin{remark}
In particular, computations of the $\op{Q}_{2d-1}$-part of the obstruction groups for splitting off hyperbolic planes from quadratic forms will follow from computations of obstruction groups for splitting off trivial lines from projective modules, cf. \cite{AsokFaselSplitting}. $\mathbb{A}^1$-homotopy group sheaves of $\op{Q}_{2d}$-part can again be related to those of $\op{Q}_{2d-1}$ via Morel's Freudenthal suspension theorem. 
\end{remark}

\section{The splitting theorems and examples}
\label{sec:splitting}

Having determined the relevant $\mathbb{A}^1$-homotopy sheaves of the orthogonal Stiefel varieties $\op{V}_{2,n}$, we can now deduce the splitting theorems and discuss examples.

\subsection{Proofs of splitting results}

The first result is a version of the splitting result of Roy, cf. \cite[Theorem 7.2]{roy}. 

\begin{theorem}
\label{thm:stablesplitting}
Let $F$ be a perfect field of characteristic unequal to $2$ and let $X=\op{Spec}A$ be a smooth affine scheme over $F$ of dimension $d$. Let $(\mathscr{P},\phi)$ be a generically trivial quadratic form over $A$ of rank $2n$ or $2n+1$. If $d\leq n-1$, then $(\mathscr{P},\phi)$ splits off a hyperbolic plane. 
\end{theorem}

\begin{proof}
Denote by $m$ the rank of the quadratic form. 
The stably hyperbolic quadratic form $(\mathscr{P},\phi)$ corresponds to a rationally trivial $\op{SO}(m)$-torsor over $X$, i.e., an element of $\op{H}^1_{\op{Nis}}(X;\op{SO}(m))$. By Theorem~\ref{thm:representability}, we can identify the latter pointed set with the set $[X,{\op{B}_{\op{Nis}}}\op{SO}(m)]_{\mathbb{A}^1}$ of $\mathbb{A}^1$-homotopy classes of maps into the classifying space of rationally trivial $\op{SO}(m)$-torsors, pointed by the trivial torsor. By Proposition~\ref{prop:stabil}, we can reformulate the claim as follows: if the rank--dimension inequality is satisfied, then the element $P\in [X,{\op{B}_{\op{Nis}}}\op{SO}(m)]_{\mathbb{A}^1}$ corresponding to $(\mathscr{P},\phi)$ is in the image of the map
\[
[X,{\op{B}_{\op{Nis}}}\op{SO}(m-2)]_{\mathbb{A}^1}\to [X,{\op{B}_{\op{Nis}}}\op{SO}(m)]_{\mathbb{A}^1}
\]
induced by the stabilization morphism.

By the relative $\mathbb{A}^1$-obstruction theory, a lift of the map $X\to{\op{B}_{\op{Nis}}}\op{SO}(m)$ exists if all of the obstructions in the cohomology groups $\op{H}^{i+1}_{\op{Nis}}(X;\bm{\pi}_i^{\mathbb{A}^1}(\op{V}_{2,m}))$ vanish  for all $i$. Since $X$ is a finite-dimensional smooth scheme, the obstruction groups will vanish automatically for $i\geq d$. The claim follows if $\bm{\pi}_i^{\mathbb{A}^1}(\op{V}_{2,m})=0$ for all $i\leq d-1$. For the case $m=2n+1$, Proposition~\ref{prop:stiefel1} states that the Stiefel varieties are $\mathbb{A}^1$-$(n-2)$-connected. Similarly, for $m=2n$, the Stiefel varieties are $\mathbb{A}^1$-$(n-2)$-connected by Proposition~\ref{prop:stiefel2}. By assumption $d\leq n-1$ which is the required range to prove the result.
\end{proof}

\begin{remark}
The splitting result in Theorem~\ref{thm:stablesplitting} is a version of a classical result analogous to Serre's splitting result for projective modules. In \cite[Theorem 7.2]{roy}, A. Roy has proved a splitting result which implies in particular that if $A$ is a commutative ring $A$ of Krull dimension $d$ in which $2$ is invertible, a stably hyperbolic quadratic form is hyperbolic if the rank of the underlying projective module is $\geq 2d+2$. For stably hyperbolic forms, Roy's result is much more general than Theorem~\ref{thm:stablesplitting} because the latter requires the commutative ring to be smooth over a perfect field. On the other hand, \cite[Theorem 7.2]{roy} needs something close to the stably hyperbolic hypothesis to get the sharp bound where we only need a generically hyperbolic quadratic form. In any case, the alternative $\mathbb{A}^1$-topological proof has the appeal that it is very close to the arguments one would use in classical topology (which inspired the algebraic proofs of Serre and Roy).
\end{remark}

Next, we will discuss what happens at the edge of the stable range. For projective modules, the first obstruction to splitting off a trivial line is given by the Euler class, cf. \cite{MField}. Something similar but slightly different happens for quadratic forms. There are again obstruction classes controlling when a generically trivial quadratic form splits off a hyperbolic plane, and they live in the cohomology of $X$ with coefficients in the first non-vanishing $\mathbb{A}^1$-homotopy sheaf of the appropriate Stiefel variety. 

\begin{theorem}
\label{thm:euler1}
Let $F$ be a perfect field of characteristic unequal to $2$ and let $X=\op{Spec}A$ be a smooth affine scheme over $F$ of dimension $d$. Let $(\mathscr{P},\phi)$ be a generically trivial quadratic form over $A$ of rank $2d+1$ admitting a spin lift. Then $(\mathscr{P},\phi)$ splits off a hyperbolic plane if and only if the obstruction class in 
\[
\op{H}^d_{\op{Nis}}(X,\mathbf{K}^{\op{MW}}_d/(1+\langle(-1)^d\rangle))\cong\left\{ \begin{array}{ll}
\widetilde{\op{CH}}^d(X)/2 & d\equiv 0\bmod 2\\
\op{H}^d_{\op{Nis}}(X,\mathbf{I}^d) & d\equiv 1\bmod 2
\end{array}\right.
\]
vanishes.
\end{theorem}

\begin{proof}
The argument is similar to the one for Theorem~\ref{thm:stablesplitting}. The relevant first obstruction lives in $\op{H}^d_{\op{Nis}}(X;\bm{\pi}_{d-1}^{\mathbb{A}^1}(\op{V}_{2,2d+1}))$, cf. Section~\ref{sec:prelims}. The relevant homotopy sheaf of $\op{V}_{2,2d+1}$ has been determined in Proposition~\ref{prop:stiefel1}. 
\end{proof}

\begin{theorem}
\label{thm:euler2}
Let $F$ be a perfect field of characteristic unequal to $2$ and let $X=\op{Spec}A$ be a smooth affine scheme over $F$ of dimension $d$. Let $(\mathscr{P},\phi)$ be a generically hyperbolic quadratic form over $A$ of rank $2d$ admitting a spin lift. Then $(\mathscr{P},\phi)$ splits off a hyperbolic plane if and only if the obstruction class in 
\[
\widetilde{\op{CH}}^d(X)\times \op{coker}\left(\op{CH}^{d-1}(X)\to\op{H}^d_{\op{Nis}}(X,\mathbf{I}^d)\right)
\]
vanishes.
\end{theorem}

\begin{proof}
By the relative $\mathbb{A}^1$-obstruction theory, the relevant obstruction lives in the group $\op{H}^d_{\op{Nis}}(X;\bm{\pi}_{d-1}^{\mathbb{A}^1}(\op{V}_{2,2d}))$ and by the dimension assumption all higher obstructions will vanish. By Proposition~\ref{prop:stiefel2}, we have $\bm{\pi}_{d-1}^{\mathbb{A}^1}(\op{V}_{2,2d})\cong \mathbf{K}^{\op{MW}}_d\times\mathbf{K}^{\op{MW}}_{d-1}$ which induces a product decomposition of the Nisnevich cohomology. For the first factor we get an Euler class group $\op{H}^d_{\op{Nis}}(X,\mathbf{K}^{\op{MW}}_d)\cong \widetilde{\op{CH}}^d(X)$ (note that we have the trivial duality here because we assumed trivial determinant). For the second factor we get the cohomology group $\op{H}^d_{\op{Nis}}(X,\mathbf{K}^{\op{MW}}_{d-1})$ which sits in an exact sequence
\[
\op{CH}^{d-1}(X)\cong\op{H}^{d-1}_{\op{Nis}}(X,\mathbf{K}^{\op{M}}_{d-1})\xrightarrow{\beta} \op{H}^d_{\op{Nis}}(X,\mathbf{I}^{d})\to \op{H}^d_{\op{Nis}}(X,\mathbf{K}^{\op{MW}}_{d-1})\to 
\op{H}^d_{\op{Nis}}(X,\mathbf{K}^{\op{M}}_{d-1}).
\]
The last term vanishes by the Gersten resolution for Milnor K-theory. This induces the required presentation. 
\end{proof}

\begin{remark}
There are natural generalizations of the above results: if the quadratic form doesn't admit a spin lift, we would have to consider cohomology equivariant with respect to the action of $\bm{\pi}^{\mathbb{A}^1}_1{\op{B}}_{\op{Nis}}\op{SO}(n)\cong\mathscr{H}^1_{\et}(\mu_2)$. Of course, understanding or computing these equivariant cohomology groups is another matter.
\end{remark}

\begin{remark}
Note that the splitting for the Stiefel variety implies in particular that the map $\op{Q}_{2d+1}\to{\op{B}}_{\op{Nis}}\op{SO}(2d+1)$ factors through the stabilization map ${\op{B}}_{\op{Nis}}\op{SO}(2d)\to{\op{B}}_{\op{Nis}}\op{SO}(2d+1)$. 
\end{remark}

\subsection{Consequences and examples: odd-rank forms} 
\label{sec:examples1}

\begin{proposition}
\label{prop:special1}
Let $F$ be an algebraically closed field of characteristic unequal to $2$ and let $X=\op{Spec} A$ be a smooth affine variety of dimension $d$ over $F$. A generically trivial quadratic form $(\mathscr{P},\phi)$ over $A$ of rank $2d+1$ which admits a spin lift splits off a hyperbolic plane. 
\end{proposition}

\begin{proof}
There are two cases. By Theorem~\ref{thm:euler1}, if $d\equiv 0\bmod 2$, the relevant obstruction lives in $\op{H}^d_{\op{Nis}}(X,\mathbf{K}^{\op{MW}}_d/2)$. The natural morphism $\mathbf{K}^{\op{MW}}_n/2\to\mathbf{K}^{\op{M}}_n/2$ induces a morphism of Gersten or Rost--Schmid complexes
\[
\xymatrix{
\bigoplus_{x\in X^{(d-1)}}\op{K}^{\op{MW}}_1(\kappa(x))/2\ar[r] \ar@{>>}[d] & 
\bigoplus_{y\in X^{(d)}}\op{GW}(\kappa(y))/2 \ar[r] \ar[d]^\cong &  \op{H}^d_{\op{Nis}}(X,\mathbf{K}^{\op{MW}}_d/2) \ar[r] \ar[d] &0 \\
\bigoplus_{x\in X^{(d-1)}}\op{K}^{\op{M}}_1(\kappa(x))/2\ar[r]  & 
\bigoplus_{y\in X^{(d)}}\mathbb{Z}/2\mathbb{Z} \ar[r] &  \op{Ch}^d(X) \ar[r] &0.
}
\]
Since  $F$ is quadratically closed, $\op{GW}(\kappa(y))/2\cong\mathbb{Z}/2\mathbb{Z}$ for all $y\in X^{(d)}$, and the left-hand vertical morphism is surjective; therefore, the right-hand vertical morphism is an isomorphism. By Roitman's theorem, $\op{CH}^d(X)$ is uniquely divisible and therefore $\op{Ch}^d(X)=0$. This implies that the obstruction group $\op{H}^d_{\op{Nis}}(X,\mathbf{K}^{\op{MW}}_d/2)$ vanishes, and by Theorem~\ref{thm:euler1}, the quadratic form splits off a hyperbolic plane. 

In the case if $d\equiv 1\bmod 2$, the relevant obstruction group is $\op{H}^d_{\op{Nis}}(X,\mathbf{I}^d)$. A similar argument as above implies that $\op{H}^d_{\op{Nis}}(X,\mathbf{I}^d)\cong \op{Ch}^d(X)\cong 0$ and again the obstruction vanishes and the quadratic form splits off a hyperbolic plane. 
\end{proof}

\begin{example}
To see that some condition is necessary in Proposition~\ref{prop:special1}, we discuss the case of the real algebraic 2-sphere $S=\mathbb{R}[X,Y,Z]/(X^2+Y^2+Z^2-1)$. By the Serre--Swan theorem, $\op{Pic}(S)=0$ and so the discussion in \cite[Section 6.3]{hyperbolic-dim3}  implies that $\op{H}^1_{\op{Nis}}(S,\op{SO}(5))\cong \op{H}^1_{\op{Nis}}(S,\op{Sp}_4)$. By \cite[Propositions 5.5, 5.13 and 5.23]{hyperbolic-dim3}, the morphism $\widetilde{\op{CH}}^2(S)\cong \op{H}^1_{\op{Nis}}(S,\op{Spin}(3))\to \op{H}^1_{\op{Nis}}(S,\op{Spin}(5))\cong \widetilde{\op{CH}}^2(S)$ induced by the stabilization map $\op{Spin}(3)\to\op{Spin}(5)$ is identified with multiplication by $2$. In particular, any class in $\widetilde{\op{CH}}^2(S)$ which is not divisible by $2$ provides an example of a rank 5 quadratic form which doesn't split off a hyperbolic plane. By \cite[Theorem 16.3.8]{fasel:memoir}, we have $\op{H}^2(S,\mathbf{I}^2)\cong \mathbb{Z}$. Using the natural surjection $\widetilde{\op{CH}}^2(S)\to \op{H}^2(S,\mathbf{I}^2)$ we find that there are classes in $\widetilde{\op{CH}}^2(S)$ which are not divisible by $2$. These correspond to quadratic forms of rank $5$ which are not obtained from rank $3$ forms by adding a hyperbolic plane. 
\end{example}

\begin{example}
To give an example for $d\equiv 1\bmod 2$, recall from \cite{hyperbolic-dim3} that for $X$ a smooth affine 3-fold over a perfect field of characteristic $\neq 2$, we have $\op{H}^1_{\op{Nis}}(X,\op{Spin}(7))\cong \op{CH}^2(X)$ with the bijection given by mapping a quadratic form or spin torsor to its second Chern class. On the other hand, we have a surjection $\op{H}^1_{\op{Nis}}(X,\op{Spin}(5))\twoheadrightarrow\widetilde{\op{CH}}^2(X)$ with the invariant given by the first Pontryagin class of the corresponding $\op{Sp}_4$-bundle. The stabilization morphism $\op{H}^1_{\op{Nis}}(X,\op{Spin}(5))\to \op{H}^1_{\op{Nis}}(X,\op{Spin}(7))$ is identified with the natural morphism $\widetilde{\op{CH}}^2(X)\to\op{CH}^2(X)$. This morphism isn't necessarily surjective, and the obstruction to surjectivity is exactly given by the Bockstein morphism $\op{CH}^2(X)\to\op{Ch}^2(X)\to \op{H}^3(X,\mathbf{I}^3)$. Incidentally, in this case, the second $\mathbb{A}^1$-homotopy sheaf of ${\op{B}}_{\op{Nis}}\op{SO}(5)$ is given $\mathbf{K}^{\op{MW}}_2$ in the extension
\[
0\to\mathbf{I}^3\to \bm{\pi}^{\mathbb{A}^1}_2({\op{B}}_{\op{Nis}}\op{Spin}(5))\to \mathbf{K}^{\op{M}}_2\to 0
\]
and the Bockstein morphism is exactly the boundary morphism for this extension. 

Consequently, the obstruction to splitting off a hyperbolic plane is exactly the third integral Stiefel--Whitney class $\beta(\overline{\op{c}}_2)$. Under the natural reduction morphism $\op{H}^3(X,\mathbf{I}^3)\to \op{Ch}^3(X)$, this class maps to $\op{Sq}^2(\overline{\op{c}}_2)=\overline{\op{c}}_3$. From the classification we had $\op{H}^1(X,\op{Spin}(7))\cong \op{CH}^2(X)$, and therefore any class in $\op{CH}^2(X)$ is realizable as second Chern class of a quadratic form. In particular, over a smooth affine 3-fold $X$, we can obtain an oriented quadratic form of rank $7$  which doesn't split off a hyperbolic plane whenever $\op{Sq}^2\colon \op{Ch}^2(X)\to\op{Ch}^3(X)$ is nontrivial. 
\end{example}

\begin{example}
The obstructions in Theorem~\ref{thm:euler1} are trivial for smooth affine quadrics of even dimension $\geq 8$ over quadratically closed fields. Any morphism $\op{Q}_{2d}\to {\op{B}}_{\op{Nis}}\op{SO}(2d+1)$ necessarily factors through the morphism ${\op{B}}_{\op{Nis}}\op{SO}(2d-1)\to {\op{B}}_{\op{Nis}}\op{SO}(2d+1)$. Viewing the morphism as element of $[\op{Q}_{2d},{\op{B}}_{\op{Nis}}\op{SO}(2d+1)]_{\mathbb{A}^1}\cong \left(\bm{\pi}^{\mathbb{A}^1}_{d-1}(\op{SO}(2d+1))\right)_{-d}$, the obstruction to lifting is simply the image under the morphism 
\[
\left(\bm{\pi}^{\mathbb{A}^1}_{d-1}(\op{SO}(2d+1))\right)_{-d}\to \left(\bm{\pi}^{\mathbb{A}^1}_{d-1}(\op{V}_{2,2d+1})\right)_{-d}
\]
induced by the projection $\op{SO}(2d+1)\to \op{V}_{2,2d+1}$. It can be checked on topological realization that this morphism is zero, hence the primary obstruction to splitting off a hyperbolic plane from a rank $2d+1$ form vanishes for the sphere $\op{Q}_{2d}$. 
\end{example}

\begin{example}
A specific example of a stably trivial quadratic form which doesn't split off a hyperbolic plane can also be obtained (but that's more a low-dimensional accident). Consider $\op{Q}_4$ over $\mathbb{R}$. By the results of \cite{hyperbolic-dim3}, the isometry classes of $\op{SO}(3)$-torsors are in bijection with $\op{GW}(F)\cong\op{H}^2(\op{Q}_4,\mathbf{K}^{\op{MW}}_2)$. The same is true for $\op{SO}(5)$-torsors. The natural stabilization map $\op{SO}(3)\to\op{SO}(5)$ induces the multiplication by 2 on homotopy sheaves, cf. \cite{hyperbolic-dim3}. Finally, the isometry classes of $\op{SO}(7)$-torsors are in bijection with $\mathbb{Z} \cong \op{H}^2(\op{Q}_4,\mathbf{K}^{\op{M}}_2)$, and the stabilization map $\op{SO}(5)\to\op{SO}(7)$ corresponds to the rank map $\dim\colon \op{GW}(F)\to\mathbb{Z}$. Now we can take an $\op{SO}(5)$-torsor over $\op{Q}_4$ corresponding to a class in $\op{GW}(\mathbb{R})$ of rank $0$ and odd signature. This torsor becomes trivial after adding a single hyperbolic plane, but it doesn't split off a hyperbolic plane because its invariant in $\op{GW}(F)$ isn't divisible by $2$. 
\end{example}

\begin{remark}
At this point, we have no precise information about the obstruction classes. An obvious guess is that the obstruction classes for $2d+1$-rank forms with $d\equiv 1\bmod 2$ are integral Stiefel--Whitney classes. In this case, the obstructions would vanish for stably trivial forms. However, looking at the conjectural structure of the first unstable $\mathbb{A}^1$-homotopy sheaf of $\op{SO}(2d+1)$, cf. the discussion in Appendix~\ref{sec:stabilization}, it's also conceivable that the obstruction classes for odd-rank forms vanish altogether. In any case, one would expect significantly stronger splitting results than those discussed in the present paper. Very likely, an investigation of the Chow--Witt rings of ${\op{B}}_{\op{Nis}}\op{SO}(2n+1)$ would be a relevant step for the identification of the obstruction class.
\end{remark}

\subsection{Consequences and examples: even-rank bundles}

Recall from Theorem~\ref{thm:euler2} that for a generically split quadratic form over a smooth affine scheme $X=\op{Spec} A$ of rank $2d$ the first obstruction to split off a hyperbolic plane lives in $\widetilde{\op{CH}}^d(X)\times \op{H}^d(X,\mathbf{W})$. We now want to discuss the relation between the Chow--Witt part of the obstruction class and the Euler class of quadratic forms as defined by Edidin and Graham in \cite{edidin:graham}. 

Let $F$ be a field of characteristic $\neq 2$, and let $V$ be an $n$-dimensional $F$-vector space. Equipping $V\oplus V^\vee$ with the quadratic form $\op{ev}(v+\phi)=\phi(v)$ given by evaluation, we can concretely realize the special orthogonal group $\op{SO}(2n)$ as the linear automorphisms of $V\oplus V^\vee$ preserving the evaluation form. If we choose a basis $e_1,\dots,e_n$ of $V$ with the corresponding dual basis $e_1^\vee,\dots,e_n^\vee$, then the evaluation form is given by 
\[
\op{ev}\colon \sum_{i=1}^n\left(\lambda_i e_i+\mu_i e_i^\vee\right)\mapsto \sum_{i=1}^n\lambda_i\mu_i.
\]
In the basis $e_i+e_i^\vee$, $e_i-e_i^\vee$ the evaluation form is given by a sum of squares of signature $(n,n)$.

Now we can write down an explicit model for the hyperbolic map $H\colon\op{GL}_n\to \op{SO}(2n)$: a linear automorphism $\alpha\colon V\to V$ in $\op{GL}_n$ maps to the corresponding automorphism $\alpha\oplus(\alpha^{-1})^\vee\colon V\oplus V^\vee\to V\oplus V^\vee$ in $\op{SO}(2n)$. For a unit-length vector $v\in V\oplus V^\vee$ (such as $e_1+e_1^\vee$), the stabilizer of $v$ in $\op{SO}(2n)$ is an orthogonal group $\op{SO}(2n-1)$ and the inclusion $\op{SO}(2n-1)\hookrightarrow\op{SO}(2n)$ is the usual stabilization morphism. 

\begin{lemma}
\label{lem:cart}
Let $F$ be a field of characteristic $\neq 2$, and let $V$ be an $n$-dimensional $F$-vector space. Then there is a cartesian square
\[
\xymatrix{
\op{GL}_{n-1}\ar[r] \ar[d] & \op{GL}_n \ar[d]^H \\
\op{SO}(2n-1)\ar[r] & \op{SO}(2n)
}
\]
where the horizontal morphisms are the usual stabilization morphisms. Moreover, the induced morphism on quotients $\op{GL}_n/\op{GL}_{n-1}\to \op{SO}(2n)/\op{SO}(2n-1)$ is an isomorphism. 
\end{lemma}

\begin{proof}
The first statement is simply the claim that $\op{GL}_{n-1}$ is the intersection of $\op{GL}_n$ and $\op{SO}(2n-1)$ inside $\op{SO}(2n)$. Note that matrices in the image of $H$ preserve the decomposition $V\oplus V^\vee$. If such a matrix preserves the vector $e_1+e_1^\vee$, then the original matrix in $\op{GL}(V)$ preserves the vector $e_1$. Moreover, because the matrix is orthogonal, it will also preserve the decomposition $V=\langle e_1\rangle\oplus \langle e_2,\dots,e_n\rangle$. This shows that the intersection is exactly $\op{GL}_{n-1}$, the endomorphisms of $\langle e_2,\dots,e_n\rangle$.
\end{proof}

Note that the quotient maps above are locally trivial in the Nisnevich topology and therefore give rise to fiber sequences, cf. \cite{gbundles2}. As a consequence of the above statement, we get a homotopy cartesian diagram of classifying spaces 
\[
\xymatrix{
{\op{B}}\op{GL}_{n-1}\ar[r] \ar[d] & {\op{B}}\op{GL}_n \ar[d] \\
{\op{B}}_{\op{Nis}}\op{SO}(2n-1)\ar[r] & {\op{B}}_{\op{Nis}}\op{SO}(2n).
}
\]

\begin{proposition}
\label{prop:special2}
Let $F$ be a field of characteristic unequal to $2$, let $X=\op{Spec} A$ be a smooth affine variety of dimension $d$ over $F$ and let $(\mathscr{P},\phi)$ be a generically split quadratic form over $A$ of rank $2d$ admitting a spin lift. 
\begin{enumerate}
\item 
Under the natural projection 
\[
\op{H}^d(X;\bm{\pi}^{\mathbb{A}^1}_{d-1}(\op{V}_{2,2d}))\cong \widetilde{\op{CH}}^d(X)\times \op{H}^d(X;\mathbf{W})\xrightarrow{\op{pr}_1} \widetilde{\op{CH}}^d(X)\to\op{CH}^d(X), 
\]
the obstruction class from Theorem~\ref{thm:euler2} maps to the Edidin--Graham Euler class of \cite{edidin:graham}. 
\item
If $F$ is quadratically closed, then $(\mathscr{P},\phi)$ splits off a hyperbolic plane if and only if its Edidin--Graham Euler class is trivial.
\end{enumerate} 
\end{proposition}

\begin{proof}
(1) We use the characterization of Euler classes for quadratic forms in \cite[Section 4, Definition 1 and Proposition 2]{edidin:graham}: if $\mathscr{V}$ is a rank $d$ vector bundle on a smooth scheme $X$ and  $\mathscr{P}:=\mathscr{V}\oplus\mathscr{V}^\vee$ the associated hyperbolic bundle with the evaluation form, then the Euler class of $(\mathscr{P},\op{ev})$ is, up to sign, the top Chern class $\op{c}_n(\mathscr{V})$. It thus suffices to prove that the obstruction class for a hyperbolic bundle reduces to the top Chern class. Now if $\mathscr{V}$ is a rank $d$ vector bundle on a smooth scheme $X$, the obstruction to split off a trivial line from $\mathscr{V}$ is the same as the obstruction for reducing the structure group for the hyperbolic bundle $(\mathscr{V}\oplus\mathscr{V}^\vee,\op{ev})$ from $\op{SO}(2d)$ to $\op{SO}(2d-1)$. This follows directly from the homotopy cartesian square of classifying spaces after Lemma~\ref{lem:cart}. Consequently, in the case of hyperbolic bundles as above, the natural projection $\op{H}^d(X;\bm{\pi}^{\mathbb{A}^1}_{d-1}(\op{V}_{2,2d}))\cong \widetilde{\op{CH}}^d(X)\times \op{H}^d(X;\mathbf{W})\xrightarrow{\op{pr}_1} \widetilde{\op{CH}}^d(X)$ maps the obstruction class of Theorem~\ref{thm:euler2} to the obstruction class for splitting off a line from $\mathscr{V}$. By \cite{AsokFaselEuler}, this is, up to a unit in the Grothendieck--Witt ring $\op{GW}(F)$, the Euler class of $\mathscr{V}$. Now under the projection $\widetilde{\op{CH}}^d(X)\to\op{CH}^d(X)$, the Euler class maps to the top Chern class, up to sign. Together with the above characterization of the Edidin--Graham Euler class, this proves the claim.

(2) This follows from Theorem~\ref{thm:euler2} and (1). Since $F$ is quadratically closed, the natural projection  $\widetilde{\op{CH}}^d(X)\xrightarrow{\cong}\op{CH}^d(X)$ is an isomorphism. The group $\op{H}^d_{\op{Nis}}(X,\mathbf{W})$ is also  trivial since we have an exact sequence 
\[
\op{CH}^{d-1}(X)\xrightarrow{\beta}\op{H}^d_{\op{Nis}}(X,\mathbf{I}^d)\to \op{H}^d_{\op{Nis}}(X,\mathbf{W})\to 0,
\] 
and $0=\op{Ch}^d(X)/2\twoheadrightarrow\op{H}^d_{\op{Nis}}(X,\mathbf{I}^d)$. In particular, the $\mathbf{W}$-cohomological contribution to the obstruction class vanishes. Then (1) shows that the obstruction class in $\widetilde{\op{CH}}^d(X)$ maps to the Edidin--Graham Euler class (up to sign) in $\op{CH}^d(X)$, which proves the claim.
\end{proof}

\begin{remark}
The above result actually means that the obstruction class reduces to the Chow--Witt-theoretic Euler class of a maximal-rank isotropic subbundle, whenever it exists. However, it is not clear that this already uniquely determines the class as a characteristic class in the Chow--Witt ring. Further analysis of the Chow--Witt rings of classifying spaces ${\op{B}}_{\op{Nis}}\op{SO}(2n)$ would be required for that.
\end{remark}

Here is a final example of the existence of quadratic bundles whose underlying vector bundles are trivial. 

\begin{theorem}
Let $F$ be a perfect field of characteristic $\neq 4$. For any $m\geq 1$, there exists an oriented quadratic form of rank $8m$ over the quadric $\op{Q}_{8m+1}$ which is generically trivial, stably nontrivial, and whose underlying vector bundle is trivial. 
\end{theorem}

\begin{proof}
By \cite{schlichting:tripathi}, we have an identification 
\[
[\op{Q}_{2n+1},{\op{B}}_{\op{Nis}}\op{SO}(\infty)]_{\mathscr{H}(F)}\cong (\mathbf{GW}^0_n)_{-n-1}(F).
\]
The result is $\op{GW}^{-n-1}_{-1}(F)\cong \op{W}^{-n}(F)$, by the identification of negative higher Grothendieck--Witt groups and Balmer's triangular Witt groups. So there are no maps unless $n\equiv 0\bmod 4$, and in this case we have $[\op{Q}_{2n+1},{\op{B}}_{\op{Nis}}\op{SO}(\infty)]_{\mathscr{H}(F)}\cong\op{W}(F)$. 
By the stabilization results, we have 
\[
[\op{Q}_{2n+1},{\op{B}}_{\op{Nis}}\op{SO}(2n)]_{\mathscr{H}(F)}\to [\op{Q}_{2n+1},{\op{B}}_{\op{Nis}}\op{SO}(\infty)]_{\mathscr{H}(F)}\to [\op{S}^{n-1}\wedge\mathbb{G}_{\op{m}}^{\wedge n+1},\op{Q}_{2n}]=0
\]
in particular we have a surjection $[\op{Q}_{2n+1},{\op{B}}_{\op{Nis}}\op{SO}(2n)]_{\mathscr{H}(F)}\twoheadrightarrow\op{W}(F)$ so that we can choose a nontrivial map $\op{Q}_{2n+1}\to{\op{B}}_{\op{Nis}}\op{SO}(2n)$ detected on $\op{W}(F)$. This corresponds to an oriented quadratic form of rank $2n$ on $\op{Q}_{2n+1}$ which is generically trivial but stably nontrivial. Now consider the composition 
\[
\op{Q}_{2n+1}\to{\op{B}}_{\op{Nis}}\op{SO}(2n)\to{\op{B}}\op{GL}_{2n},
\]
which corresponds to taking the underlying vector bundle. This is way inside the stable range, so we have $[\op{Q}_{2n+1},{\op{B}}\op{GL}_{2n}]\cong (\mathbf{K}^{\op{Q}}_n)_{-n-1}=0$, i.e., the quadratic form we have ``constructed'' has trivial underlying vector bundle. 
\end{proof}

\begin{remark}
This is an algebraic version of the example in the answer of David Chataur to MathOverflow question 112764 of a real vector bundle which is stably nontrivial but has trivial characteristic classes. The above example would have all the characteristic classes in $\op{CH}^\bullet({\op{B}}\op{O}(2n))$ trivial: either use the fact that these classes are generated by Chern classes of the underlying vector bundle or use the fact that $\op{CH}^{>0}(\op{Q}_{2n+1})=0$ since $\op{Q}_{2n+1}\cong\mathbb{A}^{n+1}\setminus\{0\}$. However, the triviality of the underlying vector bundle also implies that Chow--Witt characteristic classes induced from ${\op{B}}_{\op{Nis}}\op{SO}(2n)\to {\op{B}}\op{GL}_{2n}$ would also all be trivial. At this point, it is not clear if all characteristic classes in $\widetilde{\op{CH}}^\bullet({\op{B}}_{\op{Nis}}\op{O}(2n))$ are induced from the underlying vector bundle. As mentioned, the questions of Chow--Witt characteristic classes of quadratic forms will be discussed elsewhere. 
\end{remark}

\appendix
\section{On stabilization sequences for orthogonal groups}
\label{sec:stabilization}

In this section, we discuss the stabilization sequences for orthogonal groups. The results concerning the $\mathbb{A}^1$-homotopy sheaves of orthogonal Stiefel varieties are used to deduce some statements concerning unstable $\mathbb{A}^1$-homotopy sheaves of the orthogonal groups. 

Recall from Section~\ref{sec:prelims} that there are $\mathbb{A}^1$-fiber sequences 
\[
\op{V}_{2,n}\to {\op{B}}_{\op{Nis}}\op{SO}(n-2)\to {\op{B}}_{\op{Nis}}\op{SO}(n). 
\]
These induce corresponding long exact sequences of $\mathbb{A}^1$-homotopy sheaves 
\[
\cdots\to \bm{\pi}^{\mathbb{A}^1}_i(\op{V}_{2,n}) \to \bm{\pi}^{\mathbb{A}^1}_i({\op{B}}_{\op{Nis}}\op{SO}(n-2))\to \bm{\pi}^{\mathbb{A}^1}_i({\op{B}}_{\op{Nis}}\op{SO}(n)) \to \bm{\pi}^{\mathbb{A}^1}_{i-1}(\op{V}_{2,n}) \to\cdots
\]
The goal of the section is to provide some more information concerning the first unstable pieces of these sequences.

\subsection{The stable range}

We begin with a discussion of the stable range.

\begin{proposition}
\label{prop:stable}
Let $F$ be a perfect field of characteristic $\neq 2$. 
\begin{enumerate}
\item The natural morphism $\bm{\pi}^{\mathbb{A}^1}_{i}(\op{SO}(2n+1))\to \bm{\pi}^{\mathbb{A}^1}_{i}(\op{SO}(\infty))\cong \mathbf{GW}^0_{i+1}$ is an isomorphism for $1\leq i\leq n-2$. 
\item The natural morphism $\bm{\pi}^{\mathbb{A}^1}_{i}(\op{SO}(2n))\to \bm{\pi}^{\mathbb{A}^1}_{i}(\op{SO}(\infty))\cong \mathbf{GW}^0_{i+1}$ is an isomorphism for $1\leq i\leq n-2$. 
\end{enumerate}
\end{proposition}

\begin{proof}
(1) The relevant piece of the stabilization sequence is
\[
 \bm{\pi}^{\mathbb{A}^1}_{i+1}(\op{V}_{2,2n+3}) \to \bm{\pi}^{\mathbb{A}^1}_{i}(\op{SO}(2n+1))\to \bm{\pi}^{\mathbb{A}^1}_{i}(\op{SO}(2n+3)) \to \bm{\pi}^{\mathbb{A}^1}_{i}(\op{V}_{2,2n+3}) 
\]
By Proposition~\ref{prop:stiefel1}, $\op{V}_{2,2n+3}$ is $\mathbb{A}^1$-$(n-1)$-connected. In particular, the stabilization morphisms are isomorphisms for $i+1\leq n-1$ which is exactly the range claimed.

(2) Similar argument. By Proposition~\ref{prop:stiefel2}, $\op{V}_{2,2n+2}$ is $\mathbb{A}^1$-$(n-1)$-connected and therefore the stabilization morphisms are isomorphisms for $i+1\leq n-1$. 
\end{proof}

\begin{remark}
The $\mathbb{A}^1$-homotopy sheaves of ${\op{B}}_{\op{Nis}}\op{SO}(\infty)$ are given by the higher Grothendieck--Witt sheaves $\mathbf{GW}^0_n$, cf. \cite{schlichting:tripathi}. Over algebraically closed fields, this reproduces the well-known classical homotopy groups of ${\op{B}}\op{O}$ as follows: for the odd-dimensional quadrics, we have 
\[
[\op{Q}_{2n-1},{\op{B}}_{\op{Nis}}\op{SO}]_{\mathscr{H}(F)}\cong\op{GW}^{-n}_{-1}(F)\cong \op{W}^{1-n}(F).
\]
By the usual identification of Balmer's triangular Witt groups, $\op{W}^{1-n}(F)=0$ for $n\not\equiv 1\bmod 4$. For $n\equiv 1\bmod 4$, these agree with the Witt groups of the field. In particular, over algebraically closed fields, this reproduces the classical pattern $\mathbb{Z}/2\mathbb{Z},0,0,0$ in the odd-degree homotopy groups of ${\op{B}}\op{O}$. Similarly, for the even-dimensional quadrics we have 
\[
[\op{Q}_{2n},{\op{B}}_{\op{Nis}}\op{SO}]_{\mathscr{H}(F)}\cong\op{GW}^{-n}_{0}(F)
\]
By \cite{schlichting}, the groups $\op{GW}^i_0(k)$ are identified with the Grothendieck--Witt groups of schemes as defined by Balmer and Walter \cite{walter}. For fields, \cite{walter} provides an identification
\[
\op{GW}^i_0(F)=\left\{\begin{array}{ll}
\op{GW}(F) & i\equiv 0\bmod 4\\
0 & i\equiv 1\bmod 4\\
\mathbb{Z} & i\equiv 2\bmod 4\\
\mathbb{Z}/2\mathbb{Z} & i\equiv 3\bmod 4\end{array}\right.
\]
Over algebraically closed fields, this reproduces the classical pattern $\mathbb{Z},\mathbb{Z}/2\mathbb{Z},\mathbb{Z},0$ in the even-degree homotopy groups of ${\op{B}}\op{O}$.\footnote{Note that the statements in \cite{schlichting} or \cite{schlichting:tripathi} refer to the \'etale classifying spaces while our interest here is in the Nisnevich classifying spaces. These classifying spaces are not $\mathbb{A}^1$-weakly equivalent (because not all quadratic forms are rationally split). However, the differences is in the connected components, both spaces have $\op{S}^1$-loop space equivalent to $\op{O}(\infty)$ and hence the higher homotopy groups of the connected component of the trivial torsor agree.} 
\end{remark}

\subsection{Stabilization in the \texorpdfstring{$B_n$}{Bn}-series}

We begin with the stabilization morphism $\op{SO}(2n-1)\to\op{SO}(2n+1)$. The relevant piece of the long exact $\mathbb{A}^1$-homotopy sequence is the following: 
\begin{eqnarray*}
\bm{\pi}^{\mathbb{A}^1}_n(\op{V}_{2,2n+1})&\to& \bm{\pi}^{\mathbb{A}^1}_{n-1}(\op{SO}(2n-1))\to \bm{\pi}^{\mathbb{A}^1}_{n-1}(\op{SO}(2n+1))\to\\\to  \bm{\pi}^{\mathbb{A}^1}_{n-1}(\op{V}_{2,2n+1})&\to& \bm{\pi}^{\mathbb{A}^1}_{n-2}(\op{SO}(2n-1))\to \bm{\pi}^{\mathbb{A}^1}_{n-2}(\op{SO}(2n+1))\to 0
\end{eqnarray*}

By the above discussion of the stable range in Proposition~\ref{prop:stable}, we have an identification $\bm{\pi}^{\mathbb{A}^1}_{n-2}(\op{SO}(2n+1))\cong\mathbf{GW}^0_{n-1}$, and  $\bm{\pi}^{\mathbb{A}^1}_{n-1}(\op{V}_{2,2n+1})$ has been determined in Proposition~\ref{prop:stiefel1}. 

The goal is to determine as much of the structure of the first unstable $\mathbb{A}^1$-homotopy sheaf $\bm{\pi}^{\mathbb{A}^1}_{n-2}(\op{SO}(2n-1))$ as we can and then compare the results with the topological situation. To determine the structure of $\bm{\pi}^{\mathbb{A}^1}_{n-2}(\op{SO}(2n-1))$, we need to understand the image of $\bm{\pi}^{\mathbb{A}^1}_{n-1}(\op{V}_{2,2n+1})$, i.e., we need to study the morphism $\bm{\pi}^{\mathbb{A}^1}_{n-1}(\op{SO}(2n+1))\to \bm{\pi}^{\mathbb{A}^1}_{n-1}(\op{V}_{2,2n+1})$. Unfortunately, this is again an unstable $\mathbb{A}^1$-sheaf of the same type. So we need to combine the stabilization sequences for $\op{SO}(2n-1)\to\op{SO}(2n+1)$ and $\op{SO}(2n+1)\to \op{SO}(2n+3)$, as follows: 
\[
\xymatrix{
\bm{\pi}^{\mathbb{A}^1}_{n}(\op{V}_{2,2n+3}) \ar[d] \ar[rd]^a\\
\bm{\pi}^{\mathbb{A}^1}_{n-1}(\op{SO}(2n+1)) \ar[r] \ar[d] & \bm{\pi}^{\mathbb{A}^1}_{n-1}(\op{V}_{2,2n+1}) \ar[r] & \bm{\pi}^{\mathbb{A}^1}_{n-2}(\op{SO}(2n-1)) 
\\
\bm{\pi}^{\mathbb{A}^1}_{n-1}(\op{SO}(2n+3)) \ar[r]_\cong & \mathbf{GW}^0_{n} \ar@{.>}[u]
}
\]

\begin{proposition}
\label{prop:unstab1}
Let $F$ be a field of characteristic $\neq 2$. Then there are exact sequences of strictly $\mathbb{A}^1$-invariant sheaves of abelian groups: 
\[
\mathbf{GW}^0_n\to \bm{\pi}^{\mathbb{A}^1}_{n-1}(\op{V}_{2,2n+1})\to \bm{\pi}^{\mathbb{A}^1}_{n-2}(\op{SO}(2n-1))\to \mathbf{GW}^0_{n-1}\to 0.
\]
\end{proposition}

\begin{proof}
We first consider the case where $n\equiv 1\bmod 2$. By Proposition~\ref{prop:stiefel1}, we can write
\[
a\colon \mathbf{K}^{\op{MW}}_{n+1}/2\cong \bm{\pi}^{\mathbb{A}^1}_{n}(\op{V}_{2,2n+3}) \to 
\bm{\pi}^{\mathbb{A}^1}_{n-1}(\op{V}_{2,2n+1})\cong \mathbf{I}^n.
\]
To identify $a$, we consider the composition $\mathbf{K}^{\op{MW}}_{n+1}\to\mathbf{K}^{\op{MW}}_{n+1}/2\xrightarrow{a} \mathbf{I}^n$ as an element in $\op{Hom}(\mathbf{K}^{\op{MW}}_{n+1},\mathbf{I}^n)$ where the Hom is taken in the category of strictly $\mathbb{A}^1$-invariant sheaves of groups. Then we have $\op{Hom}(\mathbf{K}^{\op{MW}}_{n+1},\mathbf{I}^n)\cong (\mathbf{I}^n)_{-n-1}(F)\cong \op{W}(F)$. In particular, any morphism $\mathbf{K}^{\op{MW}}_{n+1}\to\mathbf{I}^n$ factors as composition of the natural projection $\mathbf{K}^{\op{MW}}_{n+1}\to\mathbf{I}^{n+1}$, the inclusion $\eta\colon \mathbf{I}^{n+1}\subset \mathbf{I}^n$ and multiplication by some element from $\op{W}(F)$. The morphism factors through $\mathbf{K}^{\op{MW}}_{n+1}/2$ if the corresponding element in $\op{W}(F)$ is 2-torsion. 

To determine the element (and consequently the morphism), we consider the $n$-fold contraction of the map. This can be described geometrically, as follows: given a morphism in 
\[
[\op{Q}_{2n}, \op{V}_{2,2n+3}]_{\mathbb{A}^1}\cong \left(\bm{\pi}^{\mathbb{A}^1}_n(\op{V}_{2,2n+3})\right)_{-n}\cong \op{K}^{\op{MW}}_1(F)/2,
\]
we can compose with the classifying morphism $\op{V}_{2,2n+3}\to{\op{B}}_{\op{Nis}}(\op{SO}(2n+1))$; then we compose the corresponding clutching morphism $\op{Q}_{2n-1}\to\op{SO}(2n+1)$ with the natural projection $\op{SO}(2n+1)\to\op{V}_{2,2n+1}$ and the corresponding map in $[\op{Q}_{2n-1},\op{V}_{2,2n+1}]_{\mathbb{A}^1}\cong \left(\bm{\pi}^{\mathbb{A}^1}_{n-1}(\op{V}_{2,2n+1})\right)_n\cong \op{W}(F)$ is the appropriate element. For the actual computation, the morphism $\op{Q}_{2n}\to\op{V}_{2,2n+3}$ factors through the inclusion $\op{Q}_{2n+1}\hookrightarrow\op{V}_{2,2n+3}$ because the composition $\op{Q}_{2n}\to \op{V}_{2,2n+3}\to \op{Q}_{2n+2}$ is null-homotopic for connectivity reasons. In particular, the composition $\op{Q}_{2n}\to \op{V}_{2,2n+3}\to {\op{B}}_{\op{Nis}}\op{SO}(2n+1)$ factors through ${\op{B}}_{\op{Nis}}\op{SO}(2n)$. The corresponding clutching map $\op{Q}_{2n-1}\to \op{SO}(2n)$ was described in Section~\ref{sec:clutching}. The composition 
\[
\op{Q}_{2n-1}\to \op{SO}(2n)\to \op{SO}(2n+1)\twoheadrightarrow \op{V}_{2,2n+1}
\]
again factors through the inclusion $\op{Q}_{2n-1}\to \op{V}_{2,2n+1}$ because connectivity forces the composition with $\op{V}_{2,2n+1}\to \op{Q}_{2n}$ to be null-homotopic. The above composition can therefore be written as $\op{Q}_{2n-1}\to \op{Q}_{2n-1}\hookrightarrow\op{V}_{2,2n+1}$ where the first morphism is an endomorphism of the sphere whose degree is $\mathbb{H}$ by Proposition~\ref{prop:degree}. In particular, the class in $[\op{Q}_{2n-1},\op{V}_{2,2n+1}]_{\mathbb{A}^1}\cong \op{W}(F)$ is trivial, irrespective of the morphism $\op{Q}_{2n}\to \op{V}_{2,2n+3}$ we started with.\footnote{As usual, the philosophy is that the above map is always induced from a morphism defined over $\mathbb{Z}[1/2]$, and there are no 2-torsion elements there.} Consequently, the stabilization sequence reduces to an exact sequence
\[
\mathbf{GW}^0_n\to\mathbf{I}^n\to \bm{\pi}^{\mathbb{A}^1}_{n-2}(\op{SO}(2n-1))\to \mathbf{GW}^0_{n-1}\to 0.
\]

In the case $n\equiv 0\bmod 2$, the relevant morphism in the diagram is 
\[
a\colon \mathbf{I}^{n+1}\cong\bm{\pi}^{\mathbb{A}^1}_n(\op{V}_{2,2n+3})\to \bm{\pi}^{\mathbb{A}^1}_{n-1}(\op{V}_{2,2n+1})\cong \mathbf{K}^{\op{MW}}_n/2.
\]
The same argument as previously, tracing throught the definition of the map using the explicit geometric description of the clutching map, implies that this morphism is zero. Consequently, we get a sequence 
\[
\mathbf{GW}^0_n\to\mathbf{K}^{\op{MW}}_n/2\to \bm{\pi}^{\mathbb{A}^1}_{n-2}(\op{SO}(2n-1))\to \mathbf{GW}^0_{n-1}\to 0.\qedhere
\]
\end{proof}

\begin{remark}
We shortly discuss the structure of the sequences over algebraically closed fields after $n$-fold contraction. 

In the case $n\equiv 1\bmod 2$, the $n$-fold contraction of the first map has the form $\op{GW}^{-n}_0(F)\to \op{W}(F)$, and $\op{GW}^{-n}_0(F)$ is either $0$ (if $i\equiv 3\bmod 4$) or $\mathbb{Z}/2\mathbb{Z}$ (if $i\equiv 1\bmod 4$). In the latter case, we can check in topological realization that the morphism must be trivial. In particular, after $n$-fold contraction, we get a short exact sequence 
\[
0\to\op{W}(F)\to \left(\bm{\pi}^{\mathbb{A}^1}_{n-2}(\op{SO}(2n-1))\right)_{-n}\to \mathbf{GW}^{-n}_{-1}\to 0.
\]

In the case $n\equiv 0\bmod 2$, the $n$-fold contraction of the first map jas the form $\op{GW}^{-n}_0(F)\to \op{GW}^{-n}_0(F)/2$, and we now have $\op{GW}^{-n}_0(F)$ is either $\op{GW}(F)$ (if $n\equiv 0\bmod 4$) or $\mathbb{Z}$ (if $n\equiv 2\bmod 4$). Again, over algebraically closed fields, we can check on topological realization that the morphism is zero; in particular, it couldn't just be the natural reduction morphism. 
\end{remark}

\begin{remark}
In the complex realization, we have the sequence 
\[
\pi_{2n-1}(\op{V}_{2,2n+1})\to \pi_{2n-2}(\op{SO}(2n-1))\to \pi_{2n-2}(\op{SO}(2n+1))\to 0.
\]
The last group is stable and the first group is $\mathbb{Z}/2\mathbb{Z}$. By the tables in \cite{kervaire}, we see that the middle group is a split extension of these two, in particular, the above sequence is also exact at the left. 

The same picture appears in the real realization: we have the sequence 
\[
\pi_{n-1}(\op{V}_{2,n+1})\to \pi_{n-2}(\op{SO}(n-1))\to \pi_{n-2}(\op{SO}(n+1))\to 0.
\]
The last group is again stable, and the first group alternates between $\mathbb{Z}/2\mathbb{Z}$ and $\mathbb{Z}$. The tables in \cite{kervaire} imply as before that the middle group is a split extension of the two outer ones. In particular, the first non-stable homotopy group of the orthogonal group is a product 
\[
\pi_{n-1}(\op{SO}(n))\cong \pi_{n-1}(\op{SO}(\infty))\times \pi_n(\op{V}_{2,n+2})
\]
for $n\geq 8$. 

One might therefore suspect that in the sequences of Proposition~\ref{prop:unstab1}, the respective morphisms 
\[
\mathbf{GW}^0_n\to\bm{\pi}^{\mathbb{A}^1}_{n-1}(\op{V}_{2,2n+1})
\]
are trivial, describing the first unstable $\mathbb{A}^1$-homotopy sheaf of the orthogonal groups $\op{SO}(2n-1)$ as an extension. Probably splitness (as in the topological case) of the sequences in the $\mathbb{A}^1$-homotopy setting would be too much to ask, but the comparison to topology would suggest that at least the corresponding cohomology operations 
\[
\op{H}^{n-1}_{\op{Nis}}(X,\mathbf{GW}^0_{n-1})\to \op{H}^{n}_{\op{Nis}}(X,\bm{\pi}^{\mathbb{A}^1}_{n-1}(\op{V}_{2,2n+1}))
\]
induced as boundary maps for the extensions are trivial. 
\end{remark}

\end{document}